\documentclass[a4paper,12pt]{amsart}
\usepackage[utf8]{inputenc}
\usepackage[T1]{fontenc}
\usepackage[UKenglish]{babel}
\usepackage[margin=18mm]{geometry}

\usepackage{color}%cyan,red;magenta,green;yellow,blue

\usepackage{graphicx}

\usepackage{amsmath,amssymb,amsfonts,amsthm}
\usepackage{mathrsfs,eucal,dsfont}
\usepackage{verbatim,enumitem}

\usepackage{hyperref,url}

\newcommand{\R}{\mathds R}

\newcommand{\I}{\mathds 1}

\def\d{{\rm d}}
\def\<{\langle}
\def\>{\rangle}

 \def\ss{\sqrt}

\def\R{\mathbb R}   \def\ss{\sqrt} 
  
  \def\vv{\varepsilon} 
\def\<{\langle} \def\>{\rangle}  
  \def\nn{\nabla}  
\def\d{\text{\rm{d}}}   
   
 \def\beq{\begin{equation}}  
 
\def\e{\text{\rm{e}}}

\def\E{\mathbb E}

\def\to{\rightarrow}
\def\8{\infty}\def\3{\triangle}
\def\1{\lesssim}

\renewcommand{\bar}{\overline}

\renewcommand{\tilde}{\widetilde}

\newtheorem{theorem}{Theorem}[section]
\newtheorem{lemma}[theorem]{Lemma}
\newtheorem{proposition}[theorem]{Proposition}
\newtheorem{corollary}[theorem]{Corollary}

\theoremstyle{definition}

\newtheorem{remark}[theorem]{Remark}

\numberwithin{equation}{section}
\begin{document}
\allowdisplaybreaks

\title[Stochastic Hamiltonian systems with $\alpha$-stable L\'evy noises] {$L^2$-exponential ergodicity of stochastic Hamiltonian systems with $\alpha$-stable L\'evy noises}

\author{
Jianhai Bao\qquad
Jian Wang}
\date{}
\thanks{\emph{J.\ Bao:} Center for Applied Mathematics, Tianjin University, 300072  Tianjin, P.R. China. \url{jianhaibao@tju.edu.cn}}

\thanks{\emph{J.\ Wang:}
School  of Mathematics and Statistics \&  Fujian Key Laboratory of Mathematical
Analysis and Applications (FJKLMAA) \&  Center for Applied Mathematics of Fujian Province (FJNU), Fujian Normal University, 350007 Fuzhou, P.R. China. \url{jianwang@fjnu.edu.cn}}

\maketitle

\begin{abstract}
 Based on the hypocoercivity approach due to Villani \cite{Villani}, Dolbeault, Mouhot and Schmeiser \cite{DMS} established  a new and simple framework to investigate directly  the $L^2$-exponential convergence to the equilibrium for the solution to the kinetic Fokker-Planck equation. Nowadays, the general framework advanced  in \cite{DMS} is named as the DMS framework for hypocoercivity.
Subsequently, Grothaus and Stilgenbauer \cite{Grothaus} builded a dual version of
the DMS framework  in the kinetic Fokker-Planck setting. No matter what the abstract DMS framework in \cite{DMS} and the dual counterpart in
\cite{Grothaus}, the densely defined linear operator involved
  is assumed
  to be decomposed into two parts, where one part is symmetric and the other part is anti-symmetric.
  Thus, the
  existing DMS framework
  is not applicable to investigate the $L^2$-exponential ergodicity for stochastic Hamiltonian systems with $\alpha$-stable L\'{e}vy noises, where one part of the
  associated infinitesimal generators is anti-symmetric whereas the other part is not symmetric. In this paper, we shall develop
 a dual version of the
 DMS framework in the  fractional kinetic Fokker-Planck setup, where one part of the densely defined linear operator under consideration need not to be symmetric. As a direct application,
 we explore  the $L^2$-exponential ergodicity of stochastic Hamiltonian systems
with $\alpha$-stable L\'evy noises. The proof is also based on Poincar\'e inequalities for non-local stable-like Dirichlet forms and the potential theory for fractional Riesz potentials.
\medskip

\noindent\textbf{Keywords:} fractional kinetic Fokker-Planck operator; stochastic Hamiltonian system
with $\alpha$-stable L\'evy noise; Poincar\'e inequality; fractional Riesz potential

\smallskip

\noindent \textbf{MSC 2020:} 60H10, 35Q84, 60J60.
\end{abstract}
\section{Introduction and Main Result}
\subsection{Background}
In physics, the Hamiltonian system, as a mathematical  formalism due to W.R. Hamilton, describes the evolution of particles in   physical systems. From the perspective on practical applications, the deterministic Hamiltonian systems   are often subject to environmental noises. Then, the environmentally perturbed system, named as the stochastic Hamiltonian system in literature, is brought into being.  So far, the stochastic Hamiltonian systems have been applied ubiquitously (see e.g. \cite{HS}) in  finance describing some risky assets, in physics portraying the synchrotron oscillations of particles in storage rings due to the impact of external fluctuating electromagnetic fields, and in   stochastic optimal control  serving as  a stochastic version of the maximum
principle of Pontryagin's type, to name a few.

With regard to the mathematical formulation, the stochastic Hamiltonian system is described by
the following degenerate stochastic differential equations (SDEs) on $\R^{2d}:=\R^d\times\R^d:$
\begin{equation}\label{D1}
\begin{cases}
\d X_t=\nn_vH(X_t,V_t)\,\d t,\\
\d V_t=-\big(\nn_x H(X_t,V_t)+F(X_t,V_t)\nn_vH(X_t,V_t)\big)\,\d t+\d Z_t,
\end{cases}
\end{equation}
where $H$ is the Hamiltonian function, $\nn_xH$ and  $\nn_vH$  stand  for the first order gradient operator with respect to the position variable $x$ and the velocity variable $v$, respectively, $F$ means the damping
coefficient,
 and $(Z_t)_{t\ge0}$ is a $d$-dimensional stochastic noise.
Throughout the paper, in some occasions, we frequently use the simplified notation $\nn$ to denote the gradient operator in case that there are no confusions evoked. In particular, when
$H(x,v)=U(x)+\frac{1}{2}|v|^2$ and $(Z_t)_{t\ge0}=(B_t)_{t\ge0}$, a $d$-dimensional standard Brownian motion,
 \eqref{D1} reduces to the stochastic damping Hamiltonian systems:
\begin{equation}\label{D2}
\begin{cases}
\d X_t= V_t \,\d t,\\
\d V_t=-(\nn  U(X_t)+F(X_t,V_t) V_t )\,\d t+\d B_t,
\end{cases}
\end{equation}
where $U$ might incorporate the confining potentials and the interaction potentials (e.g., the Lennard-Jones potential and the Coulomb potential).

In the past few years, great progresses have been made on the ergodicity of stochastic Hamiltonian systems \eqref{D2} with regular potentials. For the polynomial-like potential $U$, the exponential ergodicity under the total variation distance was addressed  in \cite{Talay,Wu} with the aid of Harris' theorem.
By making use of the mixture of the reflection coupling  and the synchronous coupling, concerning the Langevin equations (i.e., \eqref{D2} with  $F(x,v)\equiv1$),
the exponential contractivity under the quasi-Wasserstin distance was tackled in \cite{EGZ}.  Recently,    the Langevin dynamics with singular potentials has also
been
received more and more attention since the interaction potentials  exhibit certain singular features. In particular, the geometric ergodicity under the total variation distance of Langevin dynamics with singular potentials has been investigated in depth via examining  Harris' theorem;
see \cite{HM} for the setting on the Lennard-Jones type interactions and \cite{LM} concerning the setup on the Coulomb interactions, respectively.

In comparison to stochastic Hamiltonian systems subject to Brownian motion  noises, the long term behavior of the counterparts environmentally perturbed by pure jump L\'{e}vy processes is rare. All the same, there are some progresses on the ergodicity of stochastic Hamiltonian systems with pure jumps in recent years. In \cite{BW},  concerning  stochastic Hamiltonian systems with pure jumps and regular potentials,  by designing a novel Markov  coupling approach,
 we dealt with the exponential ergodicity  under  the multiplicative Wasserstein type distance.
 In the meantime, based on distinctive  constructions of Lyapunov functions and the Hörmander theorem for non-local operators, the exponential ergodicity under the total variation distance was explored  \cite{BFW} via Harris' theorem     for L\'{e}vy driven Langevin dynamics, where
 the singular potentials might be the Coulomb potentials or the
 Lennard-Jones-like potentials.

\ \

Besides the exponential ergodicity under the total variation or the Wasserstein type distance, there are plenty  of works that are devoted to the $L^2$-exponential ergodicity. We recall some facts related to it.
Let  $(X_t)_{t\ge0}$ be a Markov process generating a Markov semigroup $(P_t)_{t\ge0}$, and the probability measure $\mu$ be an invariant probability measure of $(P_t)_{t\ge0}$. The Markov process $(X_t)_{t\ge0}$ is called $L^2$-exponentially  ergodic, if there exist constants $c,\lambda>0$ such that for all $f\in L^2(\mu)$ and $t>0$,
\begin{equation}\label{D00}
\mbox{Var}_\mu(P_tf)\le c\,\e^{-\lambda t}\mbox{Var}_\mu(f),
\end{equation}
where $ \mbox{Var}_\mu(f):=\mu(f^2)-\mu(f)^2$ with $\mu(f):=\int f\d \mu.$ So far, the $L^2$-exponential ergodicity above has multiple applications.
For instance, the explicit bounds involved in  \eqref{D00} may provide insights into effectiveness of stochastic algorithms. In particular,
the explicit constants $c,\lambda>0$ in \eqref{D00} furnish an upper bound on the integrated autocorrelation, which indeed is  a performance measure of Monte Carlo estimators; see, for instance, \cite{ADNR}. On the other hand, the $L^2$-exponential ergodicity implies characterization of convergence to equilibrium in the other regimes; see \cite[Chapter 8]{Chen} for a very nice diagram of nine types of ergodicity.

For symmetric Markov processes, one of the powerful tools to investigate ergodicity (under, for example,  variance or relative entropy)
is the functional inequality (e.g., the Poincar\'{e} type inequality and the log-Sobolev
inequality). Concerned with a symmetric Markov process under investigation, the corresponding Markov semigroup is $L^2$-exponentially decay
once the associated Poincar\'{e} inequality is valid;
 see \cite[Theorem 1.1.1, p.24]{Wang} or \cite[Theorem 4.2.5, p.183]{BGL}. Whereas, as far as  non-symmetric Markov processes are concerned, the situation will be drastically different.
To demonstrate this aspect, we focus on \eqref{D1} with $H(x,v)=U(x)+\Phi(v)$ for some smooth functions $U$ and $\Phi$, $F(x,v)\equiv1$, and $(Z_t)_{t\ge0}=(B_t)_{t\ge0}$, a $d$-dimensional Brownian motion; that is,
\begin{equation}\label{E0-}
\begin{cases}
\d X_t=  \nn \Phi(V_t)  \,\d t,\\
\d V_t=-\big(\nn U(X_t)+ \nn\Phi(V_t) \big)\,\d t+\d B_t.
\end{cases}
\end{equation}
 If the prerequisite $C_{U,\Phi}:=\iint_{\R^d\times\R^d}\e^{-(U(x)+\Phi(v))}\,\d x\,\d v<\8$ holds true, then the probability measure
\begin{equation}\label{E1-}
\mu(\d x,\d v):= C_{U,\Phi}^{-1}\e^{-(U(x)+\Phi(v))}\,\d x\,\d v
\end{equation}
is an invariant probability measure of the Markov semigroup $(P_t)_{t\ge0}$ corresponding to the Markov process $(X_t,V_t)_{t\ge0}$. Due to the invariance of $\mu$, we have
 \begin{equation*}
\partial_t\mu((P_tf)^2)=-2\mu(\Gamma(P_tf)),
\end{equation*}
 where   $\Gamma (f)=\frac{1}{2}|\nn _vf|^2$ is the Carr\'{e} du champ operator;  see \cite[p.20-22 \& p.122-125]{BGL}. If there exists a constant $c_0>0$ such that
 the Poincar\'{e} inequality:
  \begin{equation}\label{EE-}
  \mbox{Var}_\mu(f)\le c_0\,\mu(\Gamma(f)),\qquad  f\in H^1(\mu)
\end{equation}
 holds true, then the $L^2$-exponential ergodicity of
  $(X_t,V_t)_{t\ge0}$ (or the semigroup $(P_t)_{t\ge0}$ is $L^2$-exponentially decay) follows from Gronwall's inequality. Nevertheless, due to $\Gamma (f)=\frac{1}{2}|\nn _vf|^2$ for $f\in H^1(\mu)$,
 the energy form $\mathscr{E}(f):=\mu(\Gamma(f))$  is reducible since the $x$-directions in $\Gamma$ are missing. Hence,  the Poincar\'{e} inequality \eqref{EE-} is not  any more valid.

Obviously,
 the infinitesimal generator of the Markov process $(X_t,V_t)_{t\ge0}$ solving \eqref{E0-} is given by
\begin{equation}\label{EE0-}
\begin{split}
(\mathscr Lf)(x,v)= &\big(\<\nn_xf(x,v),\nn \Phi(v)\>-\<\nn_vf(x,v),\nn U(x)\>\big)\\
&+\Big(  -\<\nn_vf(x,v),\nn \Phi(v)\>+\frac{1}{2}\Delta_vf(x,v) \Big)\\
=&:(\mathscr L_af)(x,v)+ (\mathscr L_sf)(x,v)
\end{split}
\end{equation}
for all $f\in C_b^2(\R^{2d})$. Under appropriate conditions imposed on the potential $U$,   the   kinetic Fokker-Planck equation corresponding to  \eqref{E0-} with  $\Phi(v)=\frac{1}{2}|v|^2$
 \begin{equation}\label{E-}
 \partial_th=\mathscr L^*h
 \end{equation}
is well posed,  where $\mathscr L^*$ is the $L^2(\mu)$-adjoint operator of $\mathscr L$. In \cite{Villani}, Villani initiated
the reputable   hypocoercivity approach, which
 has been applied successfully in coping  with exponential convergence of the solution $h$ to \eqref{E-}
 in the $H^1(\mu)$-sense, in the $L^2(\mu)$-sense, and in the relative entropy sense, respectively.
Particularly, in order to obtain the $L^2$-exponential convergence, an additional  $L^2$-gradient estimate need to be provided; see, for example, \cite[Remark 3.3]{BGH}. Later, based on a crucial source of inspiration from \cite{Herau},
 Dolbeault, Mouhot, and Schmeiser \cite{DMS} established  a new and simple framework to investigate directly  the $L^2$-exponential convergence of the solution $h$  to \eqref{E-} by examining conveniently   coercivity inequalities, an algebraic relation, and boundedness of auxiliary operators.
In comparison with the  hypocoercivity strategy in \cite{Villani}, the outstanding feature of the abstract setting advanced in \cite{DMS}
lies  in its succinctness and directness, and, most importantly,  bypassing the examination of  the $L^2$-gradient estimate in short time.
Nowadays, the general framework developed  in \cite{DMS} is named as the DMS framework for hypocoercivity in the literature.
Subsequently, the DMS framework  in the Fokker-Planck setting  was extended further in \cite{Grothaus}  to study the long-time behavior of strongly continuous semigroups generated by Kolmogorov backward operators. In particular, as an important application, the $L^2$-exponential ergodicity of the degenerate spherical
velocity Langevin equation was handled  in \cite{Grothaus}.
See \cite{CLW} for the recent study on more refined explicit estimates of exponential decay rate of underdamped Langevin
dynamics in the $L^2$-distance.
Meanwhile, the authors in \cite{GW}  went a step further to generalize
the DMS general framework and to tackle the $L^2$-algebraic ergodicity of \eqref{E0-}.  Additionally, \cite{ADNR} and \cite{ADW}
formulated a symmetrization-antisymmetrization version of the DMS setup so that the geometric hypocoercivity and the subgeometric hypocoercivity for piecewise-deterministic Markov process Monte Carlo methods can be established.
Also see \cite{LW} for explicit $L^2$-exponential convergence rates for a class of piecewise deterministic-Markov processes for sampling.

Regardless of the abstract DMS framework in \cite{DMS,Villani} and the dual counterpart in
\cite{Grothaus,GW}, the densely defined linear operator $\mathscr L$ involved, which generates a strongly continuous $C_0$-contractive semigroup, is assume to be decomposed into two parts, where one part is symmetric and the other part is anti-symmetric.
In \eqref{EE0-}, $\mathscr L_a$ is $L^2(\mu)$-antisymmetric and $\mathscr L_s$ is $L^2(\mu)$-symmetric
so that the DMS setups in \cite{Grothaus} and \cite{GW}  are applicable to investigate  the $L^2$-exponential ergodicity and the   $L^2$-subexponential ergodicity  of the Markov semigroup associated with \eqref{E0-}, respectively.

\subsection{Setting} As mentioned before, in certain scenarios, the deterministic Hamiltonian systems are   influenced  by random fluctuations with discontinuous sample paths rather than continuous sample paths.  In such context, the $d$-dimensional noise process $(Z_t)_{t\ge0}$ can be modeled naturally by a pure jump L\'{e}vy process (for example, a symmetric $\alpha$-stable process) so the formulation \eqref{E0-} need to be modified accordingly. More precisely, replacing the Brownian motion   $(B_t)_{t\ge0}$  by
 a symmetric $\alpha$-stable process $(L_t)_{t\ge0}$ enables us to reformulate \eqref{E0-}    as below:
\begin{equation}\label{E--}
\begin{cases}
\d X_t=  \nn \Phi(V_t) \, \d t,\\
\d V_t=-\nn U(X_t)\,\d t- \nn\Phi(V_t)  \,\d t+\d L_t.
\end{cases}
\end{equation}

\ \

Superficially, there are no essential distinctions  between the SDE \eqref{E0-} and the SDE \eqref{E--}  by changing merely noise patterns.
Whereas, plenty of intrinsic  changes are to be encountered. Most importantly, the probability measure $\mu$ introduced in \eqref{E1-} is no longer an invariant probability measure of the Markov process $(X_t,V_t)_{t\ge0}$ solving \eqref{E--}. Concerning SDEs with jumps even for non-degenerating cases,
 the problem on
solving the explicit expressions of invariant probability measures is a tough task and is impossible for
almost all of scenarios. This is the prime issue we must be confronted with when we explore the $L^2$-exponential ergodicity for stochastic Hamiltonian systems with L\'{e}vy noises.
In spite of this, it is still possible to figure out the closed form of invariant probability measures for jump diffusions with special structures; see, for instance, \cite{HMW,SZTG} for related details.
To guarantee that $\mu$ given in  \eqref{E1-} is still an invariant probability measure, we need to alter  the drift term of \eqref{E--} in a suitable manner.  There are several different ways to
amend the drift term so that our purpose can be achieved. One of the  potentials   is that the drift term   $\nn \Phi$  in the position component is kept while the drift part $-\nn \Phi$ in the velocity component is substituted  by
\begin{equation}\label{W}
b_\Phi(v):=\e^{\Phi(v)}\nn \big((-\Delta)^{{\alpha}/{2}-1}\e^{-\Phi(v)}\big),\quad v\in\R^d,
\end{equation}
when $d>2-\alpha$. Herein, $(-\Delta)^{{\alpha}/{2}-1}$ is the fractional Laplacian operator defined via the inverse of the Riesz potential (see e.g. \cite[Definition 2.11]{K}).
In this context, \eqref{E--}
can be  rewritten as
\begin{equation}\label{E1}
\begin{cases}
\d X_t= \nn \Phi(V_t) \d t,\\
\d V_t=\big(-\nn U(X_t)+b_\Phi(V_t)\big)\d t+\d L_t.
\end{cases}
\end{equation}
The detail that $\mu$  defined by  \eqref{E1-}
is an invariant probability measure of $(X_t,V_t)_{t\ge0}$ determined by \eqref{E1} will be elaborated in Lemma \ref{IPM} below. In fact, given the local equilibrium $F(v)=\e^{-\Phi(v)}$, the friction force $b_\Phi$ defined by \eqref{W} is the solution to the fractional Fokker-Planck equation:
\begin{equation*}
(-\Delta)^{\alpha/2}F+\mbox{div}_v(b_\Phi F)=0.
\end{equation*}
See e.g. \cite[p.1048]{BDL} for related details. Note that, when $\alpha=2$, it is easy to see that $b_\Phi$ defined by \eqref{W} is reduced into $-\nn\Phi(v)$, which coincides with the counterpart in the Brownian motion setting.
In addition, \cite{SZTG} provided    another alternative of the drift term $b_\Phi$, where the $i$-th component $b_{\Phi,i}$ is given by
\begin{equation}\label{W1}
 b_{\Phi,i}(v):=\e^{\Phi(v)}\mathcal D_{v_i}^{\alpha-2}(\e^{-\Phi(v)}\partial_i \Phi(v)),
\end{equation}
where $\mathcal D$ means the fractional
Riesz derivative defined via the Fourier transform and the inverse Fourier transform, and $\partial_i$ stands for the
 partial derivative with respect to the $i$-th component $v_i$. In comparison to $b_\Phi$ defined in \eqref{W1}, $b_\Phi$,   introduced in \eqref{W},  is much more explicit.  Based on this point of view, in this work we are interested in the stochastic Hamiltonian system \eqref{E1}, where $b_\Phi$ is defined in \eqref{W} rather than in \eqref{W1}.

\ \

The infinitesimal generator $\mathscr L$ of $(X_t,V_t)_{t\ge0}$ solving \eqref{E1} is given by,  for all $f\in C_b^2(\R^{2d})$,
\begin{equation}\label{W2}
\begin{split}
(\mathscr Lf)(x,v)&=\big(\<\nn \Phi(v),\nn_xf(x,v)\>-\<\nn U(x),\nn_vf(x,v)\>\big)\\
&\quad+\big(\<b_\Phi(v),\nn_vf(x,v)\>-(-\Delta_v )^{{\alpha}/{2}}f(x,v)\big)\\
&=:(\mathscr L_0f)(x,v)+(\mathscr L_1f(x,\cdot))(v),
\end{split}
\end{equation}
where for any $g\in C_b^2(\R^d)$,
\begin{equation}\label{D0}
(\mathscr L_1g)(v):=\<b_\Phi(v),\nn g(v)\>-(-\Delta  )^{{\alpha}/{2}}g( v).
\end{equation}
By the chain rule, we find that $\mathscr L_0$ is $L^2(\mu)$-antisymmetric while $\mathscr L_1$ is not $L^2(\mu)$-symmetric so
the DMS abstract framework \cite{Grothaus} is not applicable to investigate the $L^2$-exponential ergodicity of \eqref{E1}. 
Therefore, another challenge in treating the $L^2$-exponential ergodicity of \eqref{E1} lies in  the nonsymmetric property of $\mathscr L_1$. To deal with the trouble brought on by the nonsymmetric property of $\mathscr L_1$, by following essentially the line of \cite{Grothaus,GW}, we shall establish an improved version of the general DMS
framework, where it is of great importance that the densely defined linear operator involved  need not to possess  a symmetric part. Once the novel  framework is available,  as an application,
 the $L^2$-exponential ergodicity of \eqref{E1} can be addressed. The detailed expositions of the
 preceding tasks will be presented sequentially in the following sections.
Finally, we  want to mention that the $L^2$-analytical properties of fractional
kinetic equations have received great interest recently. In particular, a new $L^2$-hypocoercivity approach has been developed in \cite{BDL} to
establish a
rate of decay compatible with the fractional diffusion limit for fractional kinetic equations without confinement. However, there are essential differences concerning the setting and the approach associated with  \cite{BDL} and the present paper. For example, the reference measure with respect to the $L^2$-hypocoercivity in \cite{BDL} is $L^2$-$(\R^{2d};\d x \d v)$ with a proper unbounded weighted function, while here we consider the $L^2$-exponential decay with respect to the invariant probability measures $\mu(\d x,\d v)$. Most importantly, the approach in \cite{BDL} is based on the fractional Nash type inequality, whereas  herein we apply the Poincar\'e inequality for non-local stable-like Dirichlet forms established in \cite{CW,WW,WJ}.

\subsection{Main result} To state our main result, we need to present assumptions on the coefficients $U$ and $\Phi$ in \eqref{E1}. First, concerning the potential $U$, we assume that
\begin{enumerate}
\item[$({\bf A}_U)$] {\it The term $U:\R^d\to\R_+$ satisfies  the following two assumptions}\,:
\begin{enumerate}
\item[(${\bf A}_{U,1}$)]{\it $U\in C^\8(\R^d;\R_+)$ is a compact function such that $x\mapsto \e^{-U(x)} $ is integrable and $x\mapsto\nn U(x)$
is a compact function; moreover,  there exist constants $c_1,c_2>0$ such that for all $x\in\R^d,$
\begin{equation}\label{D10}
\|\nn^2U(x)\|\le c_1|\nn U(x)|+c_2,
\end{equation}
where $\nn^2$ means the second order gradient operator and $\|\cdot\|$ denotes the operator norm.}
\item[(${\bf A}_{U,2}$)]\it  $$\liminf_{|x|\to\infty}\frac{\langle \nabla U(x),x\rangle}{|x|}>0.$$
\end{enumerate}
\end{enumerate}

With regard to  the function $\Phi$, we assume that

\begin{enumerate}
\item[$({\bf A}_{\Phi})$] \it The term $\Phi:\R^d\to\R_+$ satisfies  the following assumptions\,{\rm:}
\begin{enumerate}
 \item[$({\bf A}_{\Phi,1})$] The function $\Phi\in C^3(\R^d;\R_+)$ is radial such that $\Phi(v)=\psi(|v|^2)$ for all $v\in \R^d$ and some $\psi\in C^3(\R_+;\R_+)$; moreover,  $|v|\mapsto \Phi(|v|)$ is non-decreasing, and $v\mapsto\e^{-\Phi(v)}$ is integrable.

 \item[$({\bf A}_{\Phi,2})$]
 $$\|\nabla \Phi\|_\infty
 +\|\nabla^3\Phi\|_\infty<\infty$$
 and $$\sup_{v\in \R^d} \left(\|\nabla^2\Phi(v)\||v|\right)<\infty,\quad \sup_{v\in \R^d} |\Phi(v)-\Phi(v/2)|<\infty,$$
 where $\nn^3$ means the third order gradient operator.
 Moreover, there exist constants $c^*,v^*>0$ such that for all $v\in\R^d$ with $|v|\ge v^*,$
\begin{equation*}
 \sup_{u\in B_1(v)}\|\nn^i\Psi (u)\|\le c^*\|\nn^i\Psi (v)\|,\quad i=1,2,3,
\end{equation*}
where $B_1(v)$ denotes the unite ball with the center $v$ and the radius $1$.

\item[$({\bf A}_{\Phi,3})$]
 $$\int_{\R^d} \frac{\e^{\Phi(v)}}{(1+|v|)^{2(d+\alpha)}}\,\d v<\infty.$$

\item[$({\bf A}_{\Phi,4})$]
$$\liminf_{|v|\to\infty} \frac{\e^{\Phi(v)}}{|v|^{d+\alpha}}>0.$$
\end{enumerate}
\end{enumerate}

Before we proceed, let's make some comments on the Assumptions (${\bf A}_U$) and (${\bf A}_\Phi$).
\begin{remark}
The Assumption (${\bf A}_{U,1}$) is imposed mainly to guarantee that the Poisson equation $(I-\mathscr L_{OD})f=h$
has a unique classical solution, where $\mathscr L_{OD}$
is  the infinitesimal generator   of overdamped Langevin dynamics
corresponding to the operator $\mathscr L_0$ given in \eqref{W2}; see Lemma \ref{lem3.5} for more details.
Regularity estimates of the solution  to the Poisson equation play a vital role in the following analysis. Concerning the
 Assumption (${\bf A}_{U,2}$), it  is one of the sufficient conditions ensuring that the $x$-marginal of the invariant probability measure $\mu$ defined in \eqref{E1-}
satisfies the Poincar\'{e}  inequality (see \eqref{E4} below); see, for instance, \cite[Corollary 1.6]{BBCG}.

The structure $\Phi(v)=\psi(|v|^2)$
 for some $\psi\in C^2(\R_+;\R_+)$, besides the uniform boundedness of $\nn \Phi$ and the integrability of the function $x\mapsto\e^{-U(x)}$, ensures that the sufficient criteria $ ({\bf H}_1)$ and $ ({\bf H}_2)$  in the abstract DMS framework is valid.
 In particular, under $({\bf A}_{\Phi,1})$ and $({\bf A}_{\Phi,2})$, it holds that $ \lim_{|v|\to\8}|\nn\e^{-\Phi(v)}|=0$,
 which
  enables us to establish a crucial link between the operators $\pi\mathscr L_0^2\pi$
 and $\mathscr L_{OD}$. This transfers the boundedness of one part for the auxiliary operator  into   estimates of the solution to the Poisson equation (see Lemma \ref{lem3.7}). Furthermore,
  the uniform boundedness and the  integrable conditions involved in $({\bf A}_{\Phi,2})$ and $({\bf A}_{\Phi,3})$ also   yields the boundedness of the other part of the auxiliary operator; see Proposition \ref{P4}  below.
 $({\bf A}_{\Phi,4})$ provides a sufficiency so that the $v$-marginal of the invariant probability measure $\mu$ in \eqref{E1-} satisfies the Poincar\'{e} inequality (see \eqref{E14} below), where the corresponding energy is a non-local stable-like Dirichlet form.
\end{remark}
The main result in the present paper is presented as below.
\begin{theorem}\label{thm1}
Assume that $d>2-\alpha$, and that both $($${\bf A}_U$$)$ and $($${\bf A}_\Phi$$)$
are satisfied. Then, the process $(X_t,V_t)_{t\ge0}$ solving \eqref{E1} is $L^2$-exponentially ergodic, i.e., there exist constants $c,\lambda>0$ such that for all $f\in L^2(\mu)$ and $t>0$,
\begin{equation*}
{\rm Var}_\mu(P_tf)\le c\,\e^{-\lambda t}{\rm Var}_\mu(f),
\end{equation*}where $(P_t)_{t\ge0}$ is the Markov semigroup generated by $(X_t,V_t)_{t\ge0}$  and $\mu$, defined in \eqref{E1-}, is an invariant probability measure of $(P_t)_{t\ge0}$.
\end{theorem}

As a direct consequence of Theorem \ref{thm1}, we have the following statement.

\begin{corollary}\label{cor}
Assume  that $($${\bf A}_U$$)$ and    $\Phi(v)=\frac{1}{2}(d+\beta)\log(1+|v|^2)$ for $\beta\in[\alpha,2\alpha) $ with
 $d>2-\alpha$.
Then,
the process $(X_t,V_t)_{t\ge0}$ solving \eqref{E1} is $L^2$-exponentially ergodic.
\end{corollary}

\begin{remark} Roughly speaking, the process $(X_t,V_t)_{t\ge0}$ solving \eqref{E0-} is $L^2$-exponentially ergodic provided that both the measures $\mu_1(\d x):= \frac{1}{C_U}\e^{-U(x)}\,\d x$ and $\mu_2(\d v): =\frac{1}{C_\Phi}\e^{-\Phi(v)}\,\d v$ satisfy the Poincar\'e inequalities; see e.g. \cite{DMS,GW,Villani}. Hence, in this sense,   the Assumption $($${\bf A}_U$$)$ is reasonable since $($${\bf A}_{U,2}$$)$ is a (mild) sufficient condition to ensure  that $\mu_1$ fulfills the Poincar\'e inequality. On the other hand, $({\bf A}_{\Phi,4})$ is a sufficient condition so that $\mu_2$ satisfies the Poincar\'e
 inequality as well. However, as mentioned above, there are a few of essential differences between \eqref{E0-} and \eqref{E1}. In particular, from the viewpoint of  infinitesimal generators, the counterpart  corresponding to \eqref{E1} (see \eqref{W2}) cannot be written into a proper form as that for \eqref{E0-},
  where  the associated infinitesimal generator is equal to the $L^2(\mu)$-antisymmetric part plus $L^2(\mu)$-symmetric part. Thus, to apply  efficiently  the dual version of the DMS framework developed here for the system \eqref{E1}, we need to take care of the $L^2$-bound for or some $L^2$-estimate of the error term for the dual operator of $\mathscr{L}_1$ in \eqref{E1}. In particular, the additional assumptions $({\bf A}_{\Phi,2})$ and $({\bf A}_{\Phi,3})$ are necessary. This in turn requires that $\beta<2\alpha$ in Corollary \ref{cor},
 which  leads to an immediate     consequence   that our main result Theorem \ref{thm1} does not work when $\mu_2$ is of (sub)-exponentially decay. The reader can refer to Remark \ref{Rem:add1} for more details on this point.
We also want to mention that such kind  conditions are  imposed commonly in
investigating the analysis properties of fractional Laplacian
 operator;   see, for example,  \cite{BDL} for the fractional hypocoercivity of kinetic equations without confinements.
\end{remark}

\ \

The remainder content of this work is organized as follows. In Section \ref{sec2}, we establish a general DMS framework, where one part of the densely defined linear operator involved is antisymmetric  while the other part need not to be symmetric. As an application, we apply the  DMS framework developed
to complete the proof of Theorem \ref{thm1}. This   will be addressed in the Section \ref{sec3}. Since the proof of Theorem \ref{thm1} is lengthy,
a series of Lemmas and propositions are prepared in  the Section \ref{sec3} so that the paper is much more  readable.

\section{A
General DMS framework}\label{sec2}
To encompass the non-local kinetic Fokker-Planck operator $\mathscr L$ defined by \eqref{W2}, in this section we aim to develop a general DMS framework. For this purpose,
some warm-up work is required to be carried out in advance.
Let $(\mathscr L,\mathcal D(\mathscr L))$ be a densely defined linear operator generating a strongly continuous contraction semigroup $(P_t)_{t\ge0}$
on a separable Hilbert space
$(H,\<\cdot,\cdot\>_H, \|\cdot\|_H)$. Assume that $\mathfrak{D} $ is a core of  $(\mathscr L,\mathcal D(\mathscr L))$,  and that $\mathscr L$  can be decomposed into:
\begin{equation*}
\mathscr L=\mathscr L_0+\mathscr L_1\quad\mbox{ on } \mathfrak{D},
\end{equation*}
where  the   linear operator $(\mathscr L_0, \mathfrak{D} )$ is   antisymmetric in $H$. Since $(\mathscr L_0, \mathfrak{D} )$ is a densely defined antisymmetric operator on $H$, $(\mathscr L_0, \mathfrak{D} )$ is a closable operator (see \cite[Lemma 26]{ADNR} or \cite[Theorem 5.1.5, p.194]{Pedersen}) with
 the  closure  $(\mathscr L_0,\mathcal D(\mathscr L_0))$. On the other hand, since the semigroup $(P_t)_{t\ge0}$ is contractive, the generator $(\mathscr L,\mathcal D(\mathscr L))$ is negative definite on $H$, i.e., $\langle -\mathscr L f,f\rangle\ge0$ for all $f\in\mathfrak D.$ Hence, the antisymmetric property of $(\mathscr L_0, \mathfrak{D} )$ yields that $\<-\mathscr L_1 f,f\>\ge0$ for all $f\in\mathfrak D.$

  Let $H_0$ be a closed subspace of $H$, so that $H$ can be written as a direct sum of $H_0$ and its orthogonal complement
$H_0^{\perp}$; see \cite[Theorem 3.3-4, p.146]{Kre}. Thus,   the orthogonal projection operator $\pi: H\rightarrow H_0$ is well defined.

\smallskip

Below, we shall assume that
\begin{enumerate}\it
\item[$($$ {\bf  H}_1$$)$]  $\mathfrak{D}   \subset\mathcal D(\mathscr L)\cap\mathcal D(\mathscr L^*)$, and $H_0\subset\{f\in \mathcal D(\mathscr L_1):\mathscr L_1f=0\} \cap\{f\in \mathcal D(\mathscr L_1^*):\mathscr L_1^*f=0\}$, where $\mathscr L^*$ and $\mathscr L_1^*$ stand for the
  adjoint operators of $\mathscr L$ and $\mathscr L_1 $ on $H,$ respectively\,$;$

\item[$($$ {\bf H}_2$$)$] $\pi\mathfrak{D}\subset\mathcal D(\mathscr L_0)$ and $\pi\mathscr L_0\pi f=0$ for all $f\in \mathfrak D; $

\item[$($$ {\bf H}_3$$)$] there exist  constants $\alpha_1,\alpha_2>0$ such that
\begin{equation}\label{EE1}
\|(I-\pi)f\|_H^2\le \alpha_1\<-\mathscr L_1 f,f\>_H,\quad f\in\mathfrak D,
\end{equation}
and
\begin{equation}\label{EE2}
\| \pi f\|_H^2\le \alpha_2\|\mathscr L_0\pi f\|_H^2,\quad f\in\mathcal D(\mathscr L_0\pi).
\end{equation}
\end{enumerate}

\smallskip

From ($ {\bf  H}_1$) and ($ {\bf  H}_2$), $H_0\subset\mathcal D(\mathscr L_1)$  and $\pi\mathfrak{D}\subset\mathcal D(\mathscr L_0)$. Thus,  for all $u\in\mathfrak D$, $\pi u\in\mathcal D(\mathscr L)=\mathcal D(\mathscr L_0)\cap\mathcal D(\mathscr L_1)$, and so $(I-\pi) u\in \mathcal D(\mathscr L)$. Consequently, the mapping $\mathfrak D\ni u\mapsto \mathscr L(I-\pi)u$ is well defined (see ($ {\bf  H}_4$) below).

Due to the fact that  $( \mathscr L_0, \mathcal D(\mathscr L_0))$ is a densely defined antisymmetric operator, in addition to
 $\pi\mathcal D(\mathscr L_0)\subset\mathcal D(\mathscr L_0)$ since  $\pi\mathfrak{D}\subset\mathcal D(\mathscr L_0)$ by ($ {\bf H_2}$) and
 $(\mathscr L_0,\mathcal D(\mathscr L_0))$ is a closure of $(\mathscr L_0, \mathfrak{D} )$,
 $(\mathscr L_0\pi, \mathcal D(\mathscr L_0))$ is a closable operator (see \cite[Lemma 26]{ADNR}) with  the closure  $(\mathscr L_0 \pi, \mathcal D(\mathscr L_0\pi))$. Because
$(\mathscr L_0\pi)^*=\pi\mathscr L_0^*=-\pi\mathscr L_0$ on $\mathfrak D$,    $(\mathscr L_0\pi)^*$ is a densely defined linear operator. This,
together with
$\mathscr L_0\pi=((\mathscr L_0\pi)^*)^*$ on $\mathfrak D$ and \cite[Theorem 5.1.5, p. 194]{Pedersen}, implies that
 $(\mathscr L_0\pi, \mathcal D(\mathscr L_0\pi))$ is a densely defined closed operator.
Next, define $G =(\mathscr L_0\pi)^* \mathscr L_0\pi $. Then,  $(G,\mathcal D(G))$ is self-adjoint and $\mathcal D(G)$ is a core of $\mathscr L_0\pi$; and moreover, for $\lambda>0$, $\lambda I+G$ is bijective from $\mathcal D(G)$ to $H$ and the inverse operator $(\lambda I+G)^{-1}$ is a self-adjoint operator with the operator norm $\|(\lambda I+G)^{-1}\|\le 1/\lambda$; see, \cite[Theorem 5.1.9 (i) and (ii), p.195]{Pedersen}.
Subsequently, the operator
\begin{equation}\label{E*}
B_\lambda:=(\lambda I+G)^{-1}(\mathscr L_0\pi)^*,\qquad \mathcal D(B_\lambda):=\mathcal D((\mathscr L_0\pi)^*)
=\mathcal D(\mathscr L_0)
\end{equation}
is well defined. Recalling from \cite[Theorem 5.1.5, p.194]{Pedersen} again that $(\mathscr L_0\pi)^*$ is a densely defined closed operator,  we deduce from \cite[Theorem 5.1.9 (iii), p.195] {Pedersen} that
\begin{equation*}
\bar B_\lambda=(\mathscr L_0\pi)^*(\lambda I+G)^{-1},\quad \mathcal D(\bar B_\lambda)=H,\quad \|\bar B_\lambda\|\le \lambda^{-{1}/{2}},
\end{equation*} where $(\bar B_\lambda ,\mathcal D(\bar B_\lambda ))$ is the closure of $(B_\lambda , \mathcal D(B_\lambda ))$.
This, along with the fact that $\pi$ is a projection operator (so $\pi^*=\pi$ and $\pi^2=\pi$) on $H$, implies $\pi \bar B_\lambda =\bar B_\lambda $.
The readers are referred to \cite[Section 2]{GW} or \cite[Appendix B]{ADNR} for related discussions on the operator $B_\lambda$ and its properties.

\smallskip

On the basis of  the preliminary materials concerned with the linear operator $B_\lambda$, we further suppose that
\begin{enumerate}\it
\item[$($$ {\bf  H}_4$$)$] $\mathfrak D\subset\mathcal D(G)$, and there exists a constant $\alpha_3:=\alpha_3(\lambda)>0$ such that for all $f\in\mathfrak D,$
 \begin{equation*}
 |\<B_\lambda \mathscr L(I-\pi)f,f\>_H|\le \alpha_3\|\pi f\|_H\|(I-\pi) f\|_H.
 \end{equation*}
\end{enumerate}

The main result in this section is stated as follows.

\begin{theorem}\label{thm}
Assume that $(  {\bf  H_1})$--$(  {\bf  H_4})$ hold true. Then, for all $f\in H$, $t>0$ and $\lambda>0$,
\begin{equation}\label{E***}
\|P_tf\|_H^2\le C\e^{-\lambda_0 t}\| f\|_H^2,
\end{equation} where
\begin{equation}
\begin{split}
\label{E**}\lambda_0:&=\frac{\vv_0}{2(1+\vv_0\lambda^{-1/2})(1+\lambda\alpha_2)},\quad  C:=\frac{\lambda^{1/2}+\vv_0 }{\lambda^{1/2}-\vv_0 },\\
 \vv_0:&= \frac{1}{2}\bigg( \lambda^{1/2} \wedge\frac{1+\lambda\alpha_2}{\alpha_1}\wedge  \frac{1}{  \alpha_1(1+\lambda\alpha_2)\alpha_3^2} \bigg)
\end{split}
\end{equation} with $\alpha_1,\alpha_2>0$ and $\alpha_3>0$  given in $($$ {\bf  H}_3$$)$ and $($$ {\bf  H}_4$$)$, respectively.
\end{theorem}

\begin{proof}
Since $(\mathscr L,\mathcal D(\mathscr L))$ is a densely defined linear operator and the associated semigroup $(P_t)_{t\ge0}$ is contractive,
it is sufficient to show that \eqref{E***} holds true for any $f\in  \mathcal D(\mathscr L)$.
Below, we define the modified entropy functional (see \cite[p. 3812]{DMS})
\begin{equation*}
I_\lambda  (f)=\frac{1}{2}\|f\|_H^2+\vv_0\<B_\lambda f,f\>_H,\quad f\in H,
\end{equation*}
where the linear operator $B_\lambda $ was defined in \eqref{E*} and $\vv_0$ was given in \eqref{E**}. Recall the basic fact that
  $f_t:=P_tf\in\mathcal D(\mathscr L)$ for  any $f\in  \mathcal D(\mathscr L)$, and notice that $ 1- \vv_0\lambda^{-1/2}\ge {1}/{2} $ in terms of the definition of $\vv_0.$
Provided that
\begin{equation}\label{E10}
\frac{\d}{\d t}I_\lambda   (f_t)
\le   -\frac{\vv_0}{2(1+\vv_0\lambda^{-1/2})(1+\lambda \alpha_2)}I_\lambda  (f_t),
\end{equation}
then \eqref{E***} with $f\in  \mathcal D(\mathscr L)$ is available  by applying Gronwall's inequality and noting
\begin{equation}\label{E13}
\frac{1}{2}\Big(1-\frac{\vv_0}{\lambda^{1/2}}\Big)\|u\|_H^2\le I_\lambda(u)\le \frac{1}{2}\Big(1+\frac{\vv_0}{\lambda^{1/2}}\Big)\|u\|_H^2,
\end{equation}
which is valid due to $$\|B_\lambda u\|_H\le \frac{1}{2\lambda^{1/2}}\|(I-\pi)u\|_H\le \frac{1}{2\lambda^{1/2}}\|u\|_H$$ for   $u\in H$ (see  \cite[Lemma 1]{DMS} or \cite[(2.7) in Lemma 2.2]{GW}   thanks to the definition of the operator $B$ and $(  {\bf  H_2})$). Therefore, it remains to prove \eqref{E10}.

By invoking the fact that $\frac{\d}{\d t}f_t=\mathscr Lf_t$ for all
$f\in\mathcal D(\mathscr L)$, we deduce that
\begin{equation}\label{EE0}
\frac{\d}{\d t}I_\lambda  (f_t)
 =\<\mathscr L f_t,f_t\>_H +\vv_0\big(\<B_\lambda \mathscr L f_t,f_t\>_H+\<B_\lambda  f_t,\mathscr Lf_t\>_H\big).
\end{equation}
Since $\mathfrak D$ is a core of $(\mathscr L,\mathcal D(\mathscr L))$,   for each fixed $t\ge0,$ there is a sequence of functions $(f_t^n)_{n\ge1}\subset\mathfrak D$ satisfying
\begin{equation}\label{EE5-}
\lim_{n\to\8}\big(\|f_t-f_t^n\|_H+\|\mathscr Lf_t-\mathscr Lf_t^n\|_H\big)=0.
\end{equation}
Concerning  the sequence $(f_t^n)_{n\ge1}\subset\mathfrak D$ above,
it is easy to see from \eqref{EE0} that
\begin{equation}\label{EE5}
\frac{\d}{\d t}I_\lambda   (f_t)=\<\mathscr L f_t^n,f_t^n\>_H +\vv_0\big( \<B_\lambda  f_t^n,\mathscr Lf_t^n\>_H+\<B_\lambda \mathscr L f_t^n,f_t^n\>_H\big)+R_t^{n,\lambda},
\end{equation}
where the remainder
\begin{equation*}
\begin{split}
R_t^{n,\lambda}:&= \<\mathscr L (f_t-f_t^n),f_t\>_H+\<\mathscr L  f_t^n ,f_t-f_t^n\>_H +\vv_0\<B_\lambda  (f_t-f_t^n),\mathscr Lf_t\>_H\\
&\quad+\vv_0\<B_\lambda   f_t^n ,\mathscr L(f_t-f_t^n)\>_H+ \vv_0\<B_\lambda \mathscr L f_t^n,f_t-f_t^n\>_H+\vv_0\<B_\lambda \mathscr L (f_t-f_t^n),f_t\>_H.
\end{split}
\end{equation*}

Taking  $\mathscr L_0^*=-\mathscr L_0$ (so $ \<\mathscr L_0 u,u\>_H =  0$ for $u\in\mathcal D(\mathscr L_0)$) and \eqref{EE1} into consideration implies that
 \begin{equation}\label{EE4}
\<\mathscr L f_t^n,f_t^n\>_H=-\<-\mathscr L_1 f_t^n,f_t^n\>_H\le -\frac{1}{\alpha_1}\|(I-\pi)f_t^n\|_H^2.
\end{equation}
Next, by  using $\mathscr L_0^*=-\mathscr L_0$ again, along with $\pi B_\lambda =B_\lambda $ and  $\mathscr L_1^*\pi u=0$ for all $u\in H$ due to ($ {\bf  H}_1$), it follows that for all $u\in\mathfrak D$,
\begin{equation*}
\<B_\lambda  u,\mathscr Lu\>_H
= \< \mathscr L_0^* B_\lambda  u,u\>_H+\<\mathscr L_1^*\pi B_\lambda  u,u\>_H=-\< \mathscr L_0 B_\lambda  u,u\>_H.
\end{equation*}
This, together with        $\|\mathscr L_0 B_\lambda u\|_H\le \|(I-\pi)u\|_H$ for all $u\in\mathfrak D$ (see \cite[Lemma 1]{DMS} or \cite[(2.8) in Lemma 2.2]{GW}), leads to
\begin{equation}\label{EE6}
|\<B_\lambda  f_t^n,\mathscr Lf_t^n\>_H|
\le \|(I-\pi)f_t^n\|_H\|f_t^n\|_H.
\end{equation}

Furthermore, according to $(  {\bf  H_1})$ and $(  {\bf  H_2})$, $\mathscr L_1\pi u=0$ and $\pi u\subset\mathcal D(\mathscr L_0)$ for all $u\in \mathfrak D$. Thus, we derive from ($ {\bf  H}_4$) that for all $u\in \mathfrak D$,
\begin{equation*}
\begin{split}
\<B_\lambda \mathscr L u,u\>_H
&=\<B_\lambda \mathscr L_0 \pi u,u\>_H+\<B_\lambda \mathscr L_1 \pi u,u\>_H+\<B_\lambda \mathscr L (I-\pi) u,u\>_H \\
&=\<B_\lambda \mathscr L_0 \pi u,u\>_H+ \<B_\lambda \mathscr L (I-\pi) u,u\>_H\\
&\le \<B_\lambda \mathscr L_0 \pi u,u\>_H+\alpha_3\|\pi u\|_H\|(I-\pi) u\|_H.
\end{split}
\end{equation*}
As a result, we   have
\begin{equation}\label{EE3}
\<B_\lambda \mathscr L f_t^n,f_t^n\>_H\le  \<B_\lambda \mathscr L_0 \pi f_t^n,f_t^n\>_H+\alpha_3\|\pi f_t^n\|_H\|(I-\pi) f_t^n\|_H.
\end{equation}
On the other hand, applying \cite[Lemma 2.3]{GW} with $A_0=
\mathscr L_0\pi$, $\alpha(r)\equiv\alpha_2$, $\Psi_0(r)\equiv0$ and $\nu(\d s)=\e^{-\lambda s}\d s$,  and taking advantage of \eqref{EE2}, $\pi^2=\pi$ and $\pi B_\lambda =B_\lambda $  yield that
\begin{equation*}
\<B_\lambda \mathscr L_0 \pi f_t^n,f_t^n\>_H=\<B_\lambda \mathscr L_0 \pi (\pi f_t^n),\pi f_t^n\>_H\le -\frac{1}{1+\lambda\alpha_2}\|\pi f_t^n\|_H^2.
\end{equation*}
Thus, plugging  this back into \eqref{EE3} gives us
\begin{equation} \label{EE7}
\<B_\lambda \mathscr L f_t^n,f_t^n\>_H\le  -\frac{1}{1+\lambda\alpha_2}\|\pi f_t^n\|_H^2+\alpha_3\|\pi f_t^n\|_H\|(I-\pi) f_t^n\|_H.
\end{equation}

Now, combining \eqref{EE4} with \eqref{EE6} and \eqref{EE7} enables us to obtain that
\begin{equation}\label{E12}
 \begin{split}
\frac{\d}{\d t}I_\lambda   (f_t)&\le-\frac{1}{\alpha_1}\|(I-\pi)f_t^n\|_H^2  -\frac{\vv_0}{1+\lambda\alpha_2}\|\pi f_t^n\|_H^2  \\
&\quad+\vv_0\big( \|(I-\pi)f_t^n\|_H\|f_t^n\|_H  +\alpha_3\|\pi f_t^n\|_H\|(I-\pi) f_t^n\|_H\big)+R_t^{n,\lambda}.
 \end{split}
\end{equation}
Since $B_\lambda $ is a bounded linear operator with the operator norm $\|B_\lambda \|\le \lambda^{-{1}/{2}}$, we deduce  from \eqref{EE5-} that  $\lim_{n\to\8}R_t^{ n,\lambda}=0$. Hence,  by letting $n\to\8$ in \eqref{E12}, \eqref{EE5-}, \eqref{EE5}  and the inequality: $ 2ab\le  \delta a^2+b^2/\delta $ for all $a,b\ge0$ and $\delta>0$ imply  that
\begin{equation}\label{EE}
 \begin{split}
\frac{\d}{\d t}I_\lambda   (f_t)&\le -\frac{1}{\alpha_1}\|(I-\pi)f_t\|_H^2  -\frac{\vv_0}{1+\lambda\alpha_2}\|\pi f_t\|_H^2  \\
&\quad+\vv_0\big( \|(I-\pi)f_t\|_H\|f_t\|_H  +\alpha_3\|\pi f_t\|_H\|(I-\pi) f_t\|_H\big)\\
&\le  -\frac{1}{ 2}\Big(\frac{1}{\alpha_1}-\vv_0\alpha_3^2(1+\lambda\alpha_2)\Big)\|(I-\pi)f_t\|_H^2-\frac{\vv_0}{2(1+\lambda\alpha_2)}\|\pi f_t\|_H^2+\frac{1}{2}\alpha_1\vv^2_0\|f_t\|_H^2.
 \end{split}
\end{equation}
According to the alternative of $\vv_0$ introduced in \eqref{E**},  we obtain that
\begin{equation*}
\frac{1}{2\alpha_1}-\vv_0\alpha_3^2(1+\lambda\alpha_2)\ge0,\quad \frac{\vv_0}{2(1+\lambda\alpha_2)}\le\frac{1}{4\alpha_1},\quad -\frac{\vv_0}{4(1+\lambda\alpha_2)} +\frac{1}{2}\alpha_1\vv^2_0\le0.
\end{equation*}
Consequently, by utilizing  the fact that $\|(I-\pi)f_t\|_H^2+\|\pi f_t\|_H^2=\|  f_t\|_H^2$, the estimate \eqref{EE} implies that
\begin{equation*}
\frac{\d}{\d t}I_\lambda   (f_t)
\le   -\frac{\vv_0}{4(1+\lambda\alpha_2)}\| f_t\|_H^2.
\end{equation*}
 Whence, \eqref{E10} follows by taking \eqref{E13} into consideration. The proof is therefore completed.
\end{proof}

Before ending this section, we make some remarks on the comparisons on
Theorem \ref{thm} and
 the DMS framework in \cite{DMS,Grothaus,GW}.
 \begin{remark}
 \begin{itemize}
 \item[{\rm (i)}] Recall that the densely defined linear operator $\mathscr L$ considered in \cite{DMS,Grothaus} has to be decomposed into the symmetric part and the antisymmetric part. Nevertheless,
 the linear operator $\mathscr L$ we focus on in this paper
is allowed to admit an antisymmetric part, but the remainder need not to be symmetric.

\item[{\rm (ii)}] In  \cite{DMS,Grothaus}, \eqref{EE1} and \eqref{EE2} in Assumption $({\bf H}_3)$ are called the microscopic coercivity and the macroscopic coercivity respectively, which are also referred to as Poincar\'{e} inequalities in \cite{GW}. Assumption ($ {\bf  H}_4$) is concerned with the boundedness of auxiliary operators.

  \item[{\rm (iii)}] Obviously, $({\bf H}_1)$ coincides with \cite[(H1)]{GW} when $\mathscr L_1$
is self-adjoint.   $({\bf H}_4)$ with $\lambda=1$ in the present paper is a little bit weaker than \cite[Assumption ({\bf H4})]{DMS}  and \cite[(H3)]{GW}. Moreover, the identity operator $I$ involved in the operator $B$ in \cite{DMS,GW} has been replaced by the operator $\lambda I$, which plays a tuneable role for our purpose.
 In particular, \cite[Assumption ({\bf H4})]{DMS} requires that, for any $f\in\mathcal D(\mathscr L)$, there exists a sequence of functions $(f_n)_{n\ge1}\subset\mathfrak D$ such that $f_n\rightarrow f$ in $H$ and $\limsup_{n\to\8}\<-\mathscr Lf_n,f_n\>\le \<-\mathscr Lf,f\>$. This condition
has been dropped in Theorem \ref{thm}. \end{itemize} \end{remark}

\section{Proof of Theorem \ref{thm1}}\label{sec3}
With the preceding general framework at hand, in this section we intend to present the proof of Theorem \ref{thm1}. Since it is a little bit
cumbersome to finish the proof of Theorem \ref{thm1}, we split the associated details
and prepare  Propositions
\ref{P1}, \ref{P3} and \ref{P4} below so that the whole proof is much more readable.
To end this, several auxiliary lemmas need to be prepared simultaneously.
We begin with the claim that the measure $\mu$ defined by \eqref{E1-} is indeed an invariant probability measure
of the stochastic system \eqref{E1} with the coefficient
$b_\Phi$ given in \eqref{W}.

\begin{lemma}\label{IPM}
Suppose that $\iint_{\R^{d}\times\R^d}\e^{-(U(x)+\Phi(v))}\,\d x\d v<\8.$  Then,  $\mu$ defined by \eqref{E1-} is an invariant probability measure
of the stochastic system \eqref{E1}.
\end{lemma}

\begin{proof}
To show that the probability measure $\mu$  defined by \eqref{E1-} is an invariant probability measure of the system \eqref{E1}, it is sufficient to verify
\begin{equation}\label{W5}
\big(\mathscr L^\dag \e^{-(U(\cdot)+\Phi(\cdot))}\big)(x,v)=0.
\end{equation}
Herein, $\mathscr L^\dag$ is
  the $L^2(\d x,\d v)$-adjoint  of the generator $\mathscr L$ associated with  the system \eqref{E1}. According to \eqref{W2}, it is easy to see that  \begin{equation*}
 \begin{split}
 (\mathscr L^\dag f)(x,v)=&\big(-\mbox{div}_x\big(\nn \Phi(v) f(x,v) \big)+\mbox{div}_v\big(f(x,v)\nn U(x)\big)\big)\\
 &+\big(-\mbox{div}_v\big(f(x,v)b_\Phi(v)\big)-(-\Delta_v )^{{\alpha}/{2}}f(x,v)\big)\big)\\
= & :(\mathscr L_0^\dag f)(x,v)+(\mathscr L_1^\dag f)(x,v),
 \end{split}
\end{equation*}
where $\mbox{div}_x$ and $\mbox{div}_v$ denote the divergence operators with respect to the $x$-variable and the $v$-variable, respectively.

Via the chain rule, we find that
 \begin{align*}
\big(\mathscr L_0^\dag \e^{-(U(\cdot)+\Phi(\cdot))}\big)(x,v)=&-\e^{-\Phi(v)}\mbox{div}_x\big(\e^{-U(x)}\nn \Phi(v)\big)+\e^{-U(x)}\mbox{div}_v\big(\e^{-\Phi(v)}\nn U(x)\big)\\
 =& \e^{-(U(x)+\Phi(v))}\<\nn U(x),\nn \Phi(v)\>- \e^{-(U(x)+\Phi(v))}\<\nn U(x),\nn\Phi(v))\>=0
\end{align*}
and that
\begin{align*}
\big(\mathscr L_1^\dag \e^{-(U(\cdot)+\Phi(\cdot))}\big)(x,v)=&-\e^{-U(x)}\big(\mbox{div}\big(\e^{-\Phi(v)}b_\Phi(v)\big)+(-\Delta )^{{\alpha}/{2}}\e^{-\Phi(v)}\big)\\
=&-\e^{-U(x)}\big(\mbox{div}\big(\nn \big((-\Delta)^{{-(2-\alpha)}/{2}}\e^{-\Phi(v)}\big)\big)+(-\Delta )^{{\alpha}/{2}}\e^{-\Phi(v)}\big)=0,
\end{align*}
where in the last equality we have taken the definition of $b_\Phi(v)$ into consideration and used the basic fact that
\begin{equation}\label{W8}
-(-\Delta)^{{\alpha}/{2}}\e^{-\Phi(v)}=\mbox{div}\big(\nn \big((-\Delta)^{-(2-{\alpha)}/{2}}\e^{-\Phi(v)}\big).
\end{equation}

Putting both equalities together, we can conclude that \eqref{W5} holds true, and so the desired assertion \eqref{W5} follows.
\end{proof}

In the following, we always assume that
$$C_U: =\int_{\R^d}\e^{-U(x)}\,\d x\in(0,\8),
\quad C_\Phi: =\int_{\R^d}\e^{-\Phi(v)}\,\d v\in(0,\8).$$
Write $\mu=\mu_1\times \mu_2$, where
 \begin{equation*}
\mu_1(\d x): =\frac{1}{C_U}\e^{-U(x)}\,\d x,\quad \mu_2(\d v): =\frac{1}{C_\Phi}\e^{-\Phi(v)}\,\d v.
\end{equation*}

In order to apply Theorem \ref{thm} to  the stochastic system \eqref{E1} with the coefficient $b_\Phi$ given in \eqref{W},
the main procedure  is to confirm all  assumptions $({\bf H}_1)$-$({\bf H}_4)$, step by step.
For this purpose, one need to specify explicitly the Hilbert space $H$, the closed subspace $H_0$, the core $\mathfrak D$ of $\mathscr L$ given by \eqref{W2},
and  the projection operator $\pi$.
More explicitly, for the invariant probability measure $\mu$ given by \eqref{E1-}, define
\begin{equation*}
H = L^2_0(\mu):=\big\{f\in L^2(\mu): \mu(f)=0\big\},
\end{equation*}
which is a Hilbert space endowed with the scalar product $\<f,g\>_2:=\mu(fg)$  and the induced norm $\|f\|_2:=\<f,f\>_2^{1/2}$
for $f,g\in L^2_0(\mu)$.
Define
\begin{equation*}
(\pi f)(x)=\int_{\R^d}f(x,v)\,\mu_2(\d v),\quad f\in L^2_0(\mu);
\end{equation*}
that is,  the velocity is drawn afresh from the marginal invariant distribution, while the position
is left unchanged. Direct calculations show that $\pi=\pi^*$ and $\pi^2=\pi$, so $\pi:L^2_0(\mu)\to H_0$  is an
orthogonal projector, where the subspace
\begin{equation*}
H_0:=\big\{f\in  L^2_0(\mu): f(x,v) \mbox{ is  independent of } v \big\}.
\end{equation*}
Let $C_b^\infty(\R^{2d})$ be the set of bounded functions on $\R^{2d}$ having bounded derivatives of any order.
Set
\begin{equation*}
C_{b,c}^\infty(\R^{2d}):=\big\{f\in C^\infty_b(\R^{2d}):(\nn_xf,\nn_vf) \mbox{ has compact support}\big\}
\end{equation*}
and
\begin{equation*}
\mathfrak D :
=L^2_0(\mu)\cap C_{b,c}^\infty(\R^{2d})=\big\{f\in C_{b,c}^\infty(\R^{2d}): \mu(f)=0\big\},
\end{equation*}
which obviously is a core of  $\mathscr L$.

In the following, the operators $\mathscr L$, $\mathscr L_0$ and $\mathscr L_1$ are given in \eqref{W2}. Let $\mathscr L^*$, $\mathscr L_0^*$ and $\mathscr L_1^*$ be the $L^2(\mu)$-adjoint operators  of $\mathscr L$, $\mathscr L_0 $ and $\mathscr L_1 $, respectively. Let $(\mathscr L,\mathcal D(\mathscr L))$, $(\mathscr L_0,\mathcal D(\mathscr L_0))$, $(\mathscr L_1,\mathcal D(\mathscr L_1))$ and $(\mathscr L_1^*,\mathcal D(\mathscr L_1^*))$ be the closures in $L^2_0(\mu)$ of  $(\mathscr L,\mathfrak D )$, $(\mathscr L_0,\mathfrak D )$, $(\mathscr L_1,\mathfrak D )$, and $(\mathscr L_1^*,\mathfrak D )$, respectively.

With the aid of all the previous preliminaries, we
prepare   the following  several propositions to complete the proof of Theorem \ref{thm1}.
\begin{proposition}\label{P1}
\it Suppose that
 $\Phi(v)=\psi(|v|^2)$ for some $\psi\in C^2(\R_+;\R_+)$.
  If $C_U,C_\Phi\in(0,\8)$ and $\mu_2(|\nn \Phi|)<\infty,$
Then,
both Assumptions $({\bf H}_1)$ and  $({\bf H}_2)$ hold true.\end{proposition}

\begin{proof} (1) {\it Examination of $({\bf H}_1)$}. By virtue of $C_U,C_\Phi\in(0,\8)$, both $\mu_1$ and $\mu_2$ are probability measures
so $\mu=\mu_1\times\mu_2$ is also a probability measure.
Recall that $\mathscr L_0^*$ and $\mathscr L_1^*$ are the $L^2(\mu)$-adjoint operators  of $\mathscr L_0 $ and $\mathscr L_1 $, respectively. Then,  $\mathscr L^*=\mathscr L_0^*+\mathscr L_1^*$. Note that   $\mathscr L_0^*=-\mathscr L_0$  so     $\mathscr L_0$  is an $L^2(\mu)$-antisymmetric operator.
By the integration by parts formula, it follows that for $f\in\mathfrak D,$
\begin{equation}\label{W10}
\begin{split}
(\mathscr L_1^*f)(x,v)=&- \e^{\Phi(v)}\mbox{div}_v\big(f(x,v)\nn \big((-\Delta)^{-(2-\alpha)/2}\e^{-\Phi(v)}\big)\big) \\
&-\e^{\Phi(v)}(-\Delta_v )^{\alpha/2}\big(f(x,v)\e^{-\Phi(v)}\big).\end{split}
\end{equation}
Obviously, $\mathfrak D\subset\mathcal D(\mathscr L)\cap\mathcal D(\mathscr L^*)$. Whence, to validate the assumption $({\bf H}_1)$,
it remains to show that
for any $f\in L_0^2(\mu)$,
\begin{equation}\label{W6}
\pi f\in \mathcal D(\mathscr L_1)\cap\mathcal D(\mathscr L^*_1),\quad (\mathscr L_1\pi f)(x,v)=(\mathscr L_1^*\pi f)(x,v)=0.
\end{equation}

 In retrospect, $(\mathscr L_1,\mathcal D(\mathscr L_1))$ and $(\mathscr L_1^*,\mathcal D(\mathscr L_1^*))$ are closed operators. Then, according to the closed graph theorem  (see \cite[Theorem 4.13-3, p. 293]{Kre}),  \eqref{W6} follows once there exists a sequence of functions  $(g_n)_{n\ge1}\subset \mathfrak D$
  so that
\begin{equation}\label{W7}
\lim_{n\to\8}\|g_n-\pi f\|_2=0,\quad (\mathscr L_1g_n)(x,v)=(\mathscr L_1^*g_n)(x,v)=0 \mbox{ for all } n\ge1.
\end{equation}
Indeed, for any $f\in L_0^2(\mu)$ (so $\mu_1(\pi f)=0$), there exists a sequence of functions $(\tilde{g}_n)_{n\ge1}\subset
C_{b,c}^\8(\R^{d})$ such that  $\mu_1(\tilde{g}_n)=0$
 and $\lim_{n\to\8}\mu_1(|\tilde{g}_n-\pi f|^2)=0$. For any $n\ge 1$, set $g_n(x,v): =\tilde{g}_n(x)$, which is independent of the velocity component.
 It is easy to see that $(g_n)_{n\ge1}\subset C_{b,c}^\8(\R^{2d})$, since $(\tilde{g}_n)_{n\ge1}\subset \mathfrak D$  with $\mu_1(\tilde{g}_n)=0$ and
 $\mu=\mu_1\times\mu_2.$
On the other hand, by making use of $\lim_{n\to\8}\mu_1(|\tilde{g}_n-\pi f|^2)=0$, taking the structure of $(g_n)_{n\ge1}$ into account, and noticing that $\mu=\mu_1\times\mu_2$ again, one can easily see that $\lim_{n\to\8}\|g_n-\pi f\|_2=0$ holds true.  Furthermore, because the designed $(g_n)_{n\ge1}$ has nothing to do with the velocity component, it follows from the definitions of $\mathscr L_1$ and $\mathscr L_1^*$ that,  $(\mathscr L_1g_n)(x,v)=0$ and
 \begin{equation*}
(\mathscr L_1^*g_n)(x,v)=- \e^{\Phi(v)}\tilde{g}_n(x)\big(\mbox{div}\big( \nn \big((-\Delta)^{-(2-\alpha)/2}\e^{-\Phi(v)}\big)\big) +(-\Delta )^{\alpha/2}(\e^{-\Phi(v)})\big)=0,
\end{equation*}
where the second identity is due to  \eqref{W8}. Consequently, the requirement \eqref{W7} is verified.

\ \
 \noindent(2) {\it Examination of $({\bf H}_2)$}.
It is obvious to see that $\pi\mathfrak D\subset\mathcal D(\mathscr L_0)$.
Furthermore, in accordance with the definitions of $\mathscr L_0$ and $\pi$, for any $f\in \mathfrak D$,
\begin{equation*}
\begin{split}
(\pi \mathscr L_0\pi f) (x)=\int_{\R^d}(\mathscr L_0\pi f)(x,v)\,\mu_2(\d v)
&=\int_{\R^d}   \< \nn \Phi(v ),\nn(\pi f)(x)\>\,  \mu_2(\d v)\\
&=2\int_{\R^d}   \psi'(|v|^2 )\< v ,\nn(\pi f)(x)\>\,  \mu_2(\d v),
\end{split}
\end{equation*}
where the last identity is due to $\Phi(v)=\psi(|v|^2)$. Then, taking advantage of $\mu_2(|\nn \Phi|)<\8$ and the rotationally invariant property of the probability measure
$\mu_2$
yields  $(\pi \mathscr L_0\pi f) (x)=0$.
Therefore, the confirmation of $({\bf H}_2)$ is complete.
\end{proof}

 Now we proceed to check the Assumption $({\bf H}_3)$. Before performing this task, we provide the explicit expression of the energy form corresponding to the symmetric operator $\mathscr L_1+\mathscr L_1^*$, where the non-local operator $\mathscr L_1$ was defined in \eqref{D0}.

 \begin{lemma}\label{lemma2} For any $f\in C^2_{b,c}(\R^d)$, it holds that
$$-\mu_2\big( f(\mathscr L_1+\mathscr L_1^*)f \big)=c_{d,\alpha}\mathscr E_{\alpha,\Phi}(f),$$ where
\begin{equation}\label{DD3}
c_{d,\alpha}:=2^\alpha\Gamma((d+\alpha)/2)/(\pi^{d/2}|\Gamma(-\alpha/2)|),\quad \mathscr E_{\alpha,\Phi}(f):=\iint_{\R^d\times\R^d}\frac{(f(v)-f(\bar v))^2}{|v-\bar v|^{d+\alpha}}\,\d v\,\mu_2(\d \bar v).
\end{equation}
\end{lemma}
\begin{proof}
It follows from the definitions of $\mathscr L_1$ and $\mathscr L_1^*$ (in particular, the definition of $b$ involved in), as well as the chain rule, that for all $f\in C_{b,c}^2(\R^d), $
\begin{equation*}
\begin{split}
\mu_2\big( f(\mathscr L_1+\mathscr L_1^*)f \big)&=\frac{1}{C_\Phi}\int_{\R^d}  \<f(v) \nn ((-\Delta)^{-(2-\alpha)/2}\e^{-\Phi(v)}),\nn f(v)\> \,\d v\\
&\quad-\frac{1}{C_\Phi}\int_{\R^d}f(v)\mbox{div}\big(f(v)\nn \big((-\Delta)^{-(2-\alpha)/2}\e^{-\Phi(v)}\big)\big)\,\d v\\
&\quad-\frac{1}{C_\Phi}\int_{\R^d}f(v)  \big( \e^{-\Phi(v)} (-\Delta )^{\alpha/2}f(v)+(-\Delta )^{\alpha/2}\big(f(v)\e^{-\Phi(v)}\big)\big)\,\d v\\
&=\frac{1}{C_\Phi}\int_{\R^d}f^2(v)(-\Delta )^{\alpha/2}\e^{-\Phi(v)}\,\d v\\
&\quad-\frac{1}{C_\Phi}\int_{\R^d}  f(v)\big( \e^{-\Phi(v)} (-\Delta )^{\alpha/2}f(v)+(-\Delta )^{\alpha/2}\big(f(v)\e^{-\Phi(v)}\big)\big)\,\d v.
\end{split}
\end{equation*}
This, together with the following two facts:
\begin{equation*}
(-\Delta )^{\alpha/2}\e^{-\Phi(v)}=c_{d,\alpha}\,{\rm p.v.} \int_{\R^d}\frac{\e^{-\Phi(v)}-\e^{-\Phi(\bar v)}}{|v-\bar v|^{d+\alpha}}\,\d \bar v
\end{equation*}
and
\begin{equation*}
\begin{split}
&\e^{-\Phi(v)} (-\Delta )^{\alpha/2}f(v)+(-\Delta )^{\alpha/2}\big(f(v)\e^{-\Phi(v)}\big)\\
&=c_{d,\alpha}\e^{-\Phi(v)}\,{\rm p.v.}\int_{\R^d}\frac{f(v)-f(\bar v)}{|v-\bar v|^{d+\alpha}}\,\d\bar v+c_{d,\alpha}\,{\rm p.v.}\int_{\R^d}\frac{f(v)\e^{-\Phi(v)}-f(\bar v)\e^{-\Phi(\bar v)}}{|v-\bar v|^{d+\alpha}}\,\d\bar v,
\end{split}
\end{equation*}
where $c_{d,\alpha}$ was defined in \eqref{DD3}, yields
\begin{equation*}
\begin{split}
\mu_2\big( f(\mathscr L_1+\mathscr L_1^*)f \big)
&=-\frac{c_{d,\alpha}}{C_\Phi}
 \iint_{\R^d\times\R^d}\frac{f(v)(f(v)-f(\bar v))}{|v-\bar v|^{d+\alpha}}\,\d v\,\e^{-\Phi(\bar v)}\,\d \bar v\\
&\quad-\frac{c_{d,\alpha}}{C_\Phi}
\iint_{\R^d\times\R^d}\frac{f(v)(f(v) -f(\bar v))}{|v-\bar v|^{d+\alpha}}\,\e^{-\Phi(v)}\,\d v\d\, \bar v.
\end{split}
\end{equation*}
Subsequently, by exchanging the variables $v$ and $\bar v$ in the second integral above, we deduce that
\begin{equation*}
\mu_2\big( f(\mathscr L_1+\mathscr L_1^*)f \big)
=-c_{d,\alpha} \iint_{\R^d\times\R^d}\frac{ (f(v)-f(\bar v))^2}{|v-\bar v|^{d+\alpha}}\,\d v\,\mu_2(\d\bar v).
\end{equation*}
Therefore, the desired assertion is proved.
\end{proof}

 With Lemma \ref{lemma2} hand, the Assumption $({\bf H}_3)$ is verifiable provided that both the marginal  $\mu_1$ and  the marginal $\mu_2$ fulfill the Poincar\'{e} inequalities. This statement is detailed in the following proposition.

\begin{proposition}\label{P3} Assume that $C_U,C_\Phi\in(0,\8)$ and
 $\Phi(v)=\psi(|v|^2)$  with  $0<\mu_2(|\nn\Phi|^2)<\infty $ for some $\psi\in C^2(\R_+;\R_+)$.
If $\mu_1$ and $\mu_2$ satisfy
the following two Poincar\'{e} inequalities respectively, i.e.,
there exist constants $c_1$ and $c_2>0$ such that
\begin{equation}\label{E4}
{\rm Var}_{\mu_1}(f)\le c_1\mu_1(|\nn f|^2),\quad f\in C_b^2(\R^d)
\end{equation}
and
\begin{equation}\label{E14}
{\rm Var}_{\mu_2}(f)\le c_2\mathscr E_{\alpha,\Phi}(f),\quad f\in C_b^2(\R^d),
\end{equation}
where $\mathscr E_{\alpha,\Phi}$ was defined in \eqref{DD3},
then Assumption $({\bf H}_3)$ holds true.
\end{proposition}

So far, there are plenty of sufficient conditions to demonstrate the  Poincar\'{e} inequality \eqref{E4}; for instance,
  Lyapunov's criterion  concerned with the generator $L=\Delta-\<\nn U,\nn\>$ in \cite[Theorem 4.6.2, p.\ 202]{BGL} and \cite[Theorem 1.4]{BBCG};
Explicit conditions on the potential term $U$, e.g.,
that there exist constants $\alpha>0$ and $R\ge0$ such that $\<x,\nn U\>\ge\alpha|x|$ for all $|x|\ge R$ in \cite[Corollary 1.6]{BBCG} or that $U$ is a convex function in \cite[Corollary 1.9]{BBCG}.
On the other hand, according to \cite[Theorem 1.1 (1) and (2)]{WW} (see also \cite{CW,WJ} for more details), if $\liminf_{|v|\to\8}{\e^{\Phi(v)}}/{|v|^{d+\alpha}}>0,$ then the Poincar\'{e} inequality \eqref{E14} is satisfied as well.

\begin{proof}[Proof of Proposition $\ref{P3}$]
  Via the standard density argument, it is sufficient to show that \eqref{EE1} and \eqref{EE2} hold respectively for all
 $f\in\mathfrak D.$  For any $f\in\mathfrak D,$ it is easy to see that $\pi f\in C_{b,c}^\8(\R^d)$. Let $\bar  f_x(v)=f(x,v)-(\pi f)(x)$ for  $(x,v)\in\R^{2d}$. It is ready to see  that
\begin{equation*}
\|(I-\pi)f\|_2^2=\mu_1(\mu_2(|{\bar  f}_\cdot|^2)).
\end{equation*}
Next, by virtue of the Poincar\'{e} inequality \eqref{E14} and Lemma \ref{lemma2}, as well as $\mu_2(\bar  f_x)=0$ for any $x\in\R^d$, we derive that for each fixed $x\in\R^d,$
 \begin{equation*}
\mu_2(|\bar  f_x|^2)=\mbox{Var}_{\mu_2}(\bar  f_x)\le c_2\mathscr E_{\alpha,\Phi}(\bar  f_x)\le c_2c_{d,\alpha}^{-1}\mu_2\big(\<-(\mathscr L_1+\mathscr L_1^*)\bar  f_x,\bar  f_x\>\big).
\end{equation*}
Then, integrating with respect to  $\mu_1(\d x)$ on both sides  and utilizing $\mu=\mu_1\times\mu_2$ yields
\begin{equation*}
\|(I-\pi)f\|_2^2=\mu_1\big(\mbox{Var}_{\mu_2}(\bar  f_\cdot)\big)\le c_2c_{d,\alpha}^{-1}\mu\big(\<-(\mathscr L_1+\mathscr L_1^*)( f-\pi f),f-\pi f\>\big).
\end{equation*}
This, together with the fact that $(\mathscr L_1\pi f)(x)=(\mathscr L_1^*\pi f)(x)=0,$ leads to
\begin{equation*}
\|(I-\pi)f\|_2^2=\mu_1\big(\mbox{Var}_{\mu_2}(\bar  f_\cdot)\big)\le 2c_2c_{d,\alpha}^{-1}\mu\big(\<- \mathscr L_1  f ,f \>\big).
\end{equation*}
Hence, we conclude that \eqref{EE1} holds true with $\alpha_1=2c_2c_{d,\alpha}^{-1}.$

 In the sequel, we still fix $f\in\mathfrak D.$ According to the definition of $\mathscr L_0$ and the fact that $\pi f$ is independent of the velocity variable, as well as that $\Phi(v)=\psi(|v|^2)$ and $\mu=\mu_1\times\mu_2$,
\begin{equation*}
\begin{split}
\|\mathscr L_0\pi f\|_2^2&=\iint_{\R^{d}\times\R^d}\<\nn \Phi(v),\nn(\pi f)(x)\>^2\,\mu(\d x,\d v) \\
&=4\sum_{i,j=1}^d\int_{\R^d}\partial_i(\pi f)(x)\partial_j(\pi f)(x)\,\mu_1(\d x)\int_{\R^d}\psi'(|v|^2)^2v_iv_j\,\mu_2(\d v),
\end{split}
\end{equation*}
where $v_i$ means the $i$-th component of $v$ and
$\partial_i:=\frac{\d}{\d x_i}$. In view of the radial  property of $h(v)=h(|v|):=\psi'(|v|^2)$ and the assumption that $\mu_2(|\nn\Phi|^2)<\infty$,
\begin{equation*}
\int_{\R^d}\psi'(|v|^2)^2v_iv_j\,\mu_2(\d v)=0, \quad i\neq j.
\end{equation*}
This, along with the symmetric property,  further results in
\begin{equation*}
\begin{split}
\|\mathscr L_0\pi f\|_2^2
&=4\sum_{i=1}^d\int_{\R^d}\big(\partial_i(\pi f)(x)\big)^2\,\mu_1(\d x)\int_{\R^d}\psi'(|v|^2)^2v_i^2\,\mu_2(\d v)\\
&=\frac{1}{d} \mu_1(|\nn(\pi f)|^2) \mu_2(|\nn \Phi|^2).
\end{split}
\end{equation*}
  Then, by invoking the precondition  $0<\mu_2(|\nn\Phi|^2)<\8$, it follows from  the Poincar\'{e} inequality \eqref{E4} that for all $f\in\mathfrak D,$
\begin{equation}\label{E16}
\mbox{Var}_{\mu_1}(\pi f)\le c_1\mu_1(|\nn(\pi f)|^2)=\frac{4c_1d}{\mu_2(|\nn\Phi |^2)}\|\mathscr L_0\pi f\|_2^2.
\end{equation}
Furthermore, the fact that $\mu_1(\pi f)=\mu(f)=0$ for $f\in\mathfrak D $ implies that
 for   $f\in\mathfrak D,$
\begin{equation}\label{E15}
\mbox{Var}_{\mu_1}(\pi f)=\mu_1((\pi f)^2)-\mu_1(\pi f)^2=\mu_1((\pi f)^2)=\mu((\pi f)^2)
\end{equation}
by noticing that $\pi f$ is not related to the velocity component and combining $\mu=\mu_1\times\mu_2.$ As a consequence, \eqref{EE2}
is verified by plugging \eqref{E15} back into \eqref{E16}.
\end{proof}

Before starting to examine the Assumption $({\bf H}_4)$, some additional work need to be implemented. The first one is to provide
an explicit expression on the operator $\pi\mathscr L_0^2\pi$, which is involved in the auxiliary operator $B_\lambda .$ To achieve this, we  recall some facts arising from the Assumption (${\bf A}_U$). For any $h\in C_b^\infty(\R^d)$, consider the Poisson equation
\begin{equation}\label{D12}
(I-\mathscr L_{OD})f=h,
\end{equation}
where
\begin{equation}\label{W9}
\big(\mathscr L_{OD}f\big)(x):=\Delta f(x)-\<\nn U(x),\nn f(x)\>,\quad f\in C_b^2(\R^d).
\end{equation}
Under the Assumption (${\bf A}_U$), in terms of \cite[Proposition 4]{CHSG},  \eqref{D12}
has a unique classical solution $f\in C_b^\8(\R^d)$, which can  be expressed explicitly via Green's formula as below:
\begin{equation*}
f(x)=\int_0^\8\e^{-s}\E h(X_s^x)\,\d s.
\end{equation*}
Herein,  $(X_t^x)_{t\ge0}$ is the solution to the  overdamped  Langevin dynamics
\begin{equation*}
\d X_t^x=-\nn U(X_t^x)\,\d t+\ss 2\,\d B_t,\quad\quad X_0^x=x,
\end{equation*}
where  $(B_t)_{t\ge0}$ is a standard $d$-dimensional Brownian motion. Throughout the paper, to emphasize the dependence on $h$, we shall write the solution  $f_h$ in lieu of $f$ to the Poisson equation \eqref{D12}.
The regularity estimates (see e.g. \cite[Lemma 2 and Proposition 5]{CHSG}) on the solution $f_h$ to the Poisson equation \eqref{D12} play a crucial role in the subsequent analysis.

\begin{lemma}\label{lem3.5}
Assume that $C_U,C_\Phi\in(0,\8)$, and suppose that
$\Phi(v)=\psi(|v|^2)$
with $\mu_2(|\nn\Phi|^2+\|\nn^2\Phi\| )<\8$
for some $\psi\in C^2(\R_+;\R_+)$ so that
 \begin{equation}\label{D3}
\lim_{|v|\to\8}|\nn\e^{-\Phi(v)}|=0.
\end{equation}
Then, for any $f\in C^2_b(\R^{2d})$,
\begin{equation}\label{D4}
(\pi\mathscr L_0^2\pi f)(x,v)= c^\star   \big((\mathscr L_{OD} \pi) f\big)(x,v),
\end{equation}
where the operator $\mathscr L_{OD}$ was defined in \eqref{W9}, and
\begin{equation}\label{D8}
c^\star:=\frac{ 2 \omega_d}{C_\Phi}\int_0^\8u^{{d}/{2}}\psi'(u)^2\e^{-\psi(u)}\,\d u
\end{equation}
with $\omega_d$ being the volume of the unit ball in $\R^d.$
\end{lemma}

\begin{proof}
According to the definition of the operator $\mathscr L_0$, we have
\begin{equation*}
(\pi\mathscr L_0^2\pi f)(x)=\int_{\R^d}\big(\<\nn \Phi(v),\nn ^2(\pi f)(x)\nn \Phi(v) \>-\<\nn U(x),\nn^2 \Phi(v)\nn (\pi f)(x)\>\big)\,\mu_2(\d v).
\end{equation*}
Note from  $\Phi(v)=\psi(|v|^2)$ that
\begin{equation}\label{D6}
\nn\Phi(v)=2\psi'(|v|^2)v,\quad \nn^2\Phi(v)=2\big(\psi'(|v|^2)\I_{d\times d}+2\psi''(|v|^2)\,v\otimes v\big).
\end{equation}
Thus, we deduce that
\begin{equation*}
\begin{split}
(\pi\mathscr L_0^2\pi f)(x)
&=4\sum_{i,j=1}^d\int_{\R^d}\psi'(|v|^2)^2v_iv_j \,\mu_2(\d v)\,\big( \nn ^2(\pi f)\big)_{ij}(x)\\
&\quad-4\sum_{i,j=1}^d\int_{\R^d} \psi''(|v|^2)v_iv_j \,\mu_2(\d v)\,\big(\nn U\big)_i(x)   \big(\nn (\pi f)\big)_j(x)\\
&\quad-2\int_{\R^d} \psi'(|v|^2) \,\mu_2(\d v)\,\<\nn U(x),   \nn (\pi f)(x)\>.
\end{split}
\end{equation*}
Furthermore, taking the radial properties of $h_1(v)=h_1(|v|):=\psi'(|v|^2)$ and $h_2(v)=h_2(|v|):=\psi''(|v|^2)$ into account, and utilizing  the rotationally invariant property of the measure $\mu_2$
 as well as $\mu_2(|\nn\Phi|^2+\|\nn^2\Phi\| )<\8$
 further yields
\begin{align*}
(\pi\mathscr L_0^2\pi f)(x)&=4\sum_{i =1}^d\int_{\R^d}\psi'(|v|^2)^2 v_i^2 \,\mu_2(\d v)\,\big( \nn ^2(\pi f)\big)_{ii}(x)\\
&\quad-4\sum_{i=1}^d\int_{\R^d}\psi''(|v|^2)v_i^2\,\mu_2(\d v)\,\big(\nn U\big)_i(x) \big(\nn (\pi f)\big)_i(x)\\
&\quad-2\int_{\R^d}\psi'(|v|^2) \,\mu_2(\d v)\,\<\nn U(x),   \nn (\pi f)(x)\>\\
&=\frac{4}{d}\int_{\R^d}\psi'(|v|^2)^2 |v|^2 \,\mu_2(\d v)\,\Delta(\pi f)(x)\\
&\quad -2 \int_{\R^d} \Big( \frac{2}{d} \psi''(|v|^2)|v|^2 + \psi'(|v|^2) \Big) \,\mu_2(\d v)  \<\nn U(x), \nn (\pi f)(x)\>.
\end{align*}
Therefore, to achieve \eqref{D4}, it is sufficient to verify
\begin{equation}\label{D5}
2 \int_{\R^d} \Big( \frac{2}{d} \psi''(|v|^2)|v|^2 + \psi'(|v|^2) \Big) \,\mu_2(\d v)=\frac{4}{d}\int_{\R^d}\psi'(|v|^2)^2 |v|^2 \,\mu_2(\d v)=c^\star<\8,
\end{equation}
where $c^\star>0$ was introduced in \eqref{D8}.

By invoking Jacobi's transformation formula, we obtain from $\mu(|\nn\Phi|^2)<\8$ that
\begin{equation*}
\frac{4}{d}\int_{\R^d}\psi'(|v|^2)^2 |v|^2 \,\mu_2(\d v)=\frac{ 4 \omega_d}{C_\Phi}\int_0^\8r^{d+1}\psi'(r^2)^2   \e^{-\psi(r^2)}\,\d r=\frac{ 2 \omega_d}{C_\Phi}\int_0^\8r^{{d }/{2} }\psi'(r )^2   \e^{-\psi(r )}\,\d r<\8.
\end{equation*}
Hence, the second equality in \eqref{D5} is provable. On the other hand,
by the integration by parts formula, it follows from \eqref{D3}   that
\begin{equation*}
\begin{split}
4\int_{\R^d}\psi'(|v|^2)^2 |v|^2 \,\mu_2(\d v)&=-\frac{1}{C_\Phi}\int_{\R^d}\<\nn\psi(|v|^2),\nn\e^{-\psi(|v|^2)}\>\,\d v\\
&=\frac{1}{C_\Phi}\int_{\R^d}\e^{-\psi(|v|^2)}{\rm trace}\big(\nn^2\psi(|v|^2)\big)\,\d v\\
&=\frac{2}{C_\Phi}\int_{\R^d}\e^{-\psi(|v|^2)}\big(d\psi'(|v|^2) +2\psi''(|v|^2)|v|^2\big)\,\d v,
\end{split}
\end{equation*}
where the last display is due to \eqref{D6}. Consequently, the first equality of \eqref{D5} is verifiable. Thus, the proof is complete.
\end{proof}

\begin{lemma}\label{lem3.7}
Assume that $({\bf A}_U)$ and  Assumptions in Lemma $\ref{lem3.5}$ hold, and suppose further that
\begin{equation}\label{W18}
\mu_2\big(|\nn\Phi|^4+\|\nn^2\Phi\|^2\big)<\8.
\end{equation}
  Then,
there exists a constant $c >0$ such that for all $f\in \mathfrak D,$
\begin{equation}\label{W23}
\|(B_{c^\star}\mathscr L_0(I-\pi))^*\pi f\|_2\le c\, \|\pi f\|_2,
\end{equation}
where $c^\star>0$ was defined in \eqref{D8}.
\end{lemma}

\begin{proof}
According to the definition of $B_\lambda$ and by virtue of the $L^2(\mu)$-antisymmetric property of $\mathscr L_0$, it follows readily from Lemma \ref{lem3.5} that for all $f\in \mathfrak D,$
\begin{equation}\label{W15}
\begin{split}
 (B_{c^\star}\mathscr L_0(I-\pi))^* \pi  f (x,v)&=-(I-\pi)\mathscr L_0^2\pi (c^\star I-\pi\mathscr L_0^2\pi)^{-1}\pi f (x,v)\\
&=-\frac{1}{c^\star}(I-\pi)\mathscr L_0^2\pi u (x,v)=-\frac{1}{c^\star}(I-\pi)\mathscr L_0^2 u (x,v),
\end{split}
\end{equation}
 where $u(x,v ):= (I-\mathscr L_{OD}\pi)^{-1}   \pi f  (x,v)$.
Note that $u(x,v)$ depends merely on  the $x$-variable  so we can write $u(x)=u(x,v)$ in what follows.
According to \cite[Proposition 4]{CHSG}, $u\in C_b^\8(\R^d)$ under the Assumption $({\bf A}_U)$.
By examining   the line to derive \eqref{D4}, we have
\begin{equation}\label{W0}
\int_{ \R^d}\<\nn\Phi(v),\nn^2u(x)\nn\Phi(v)\>\,\mu_2(\d v)=c^\star\Delta u(x)
\end{equation}
and
\begin{equation*}
\int_{ \R^d}\<\nn U(x),\nn^2\Phi(v)\nn u(x)\>\,\mu_2(\d v)=c^\star \<\nn U(x),\nn u(x)\>.
\end{equation*}
Then,  along with   the definition of $\mathscr L_0$, we infer from \eqref{W15} that
\begin{equation*}
\begin{split}
& (B_{c^\star}\mathscr L_0(I-\pi))^* \pi  f (x,v)\\
&=-\frac{1}{c^\star}\bigg(\<\nn\Phi(v),\nn^2u(x)\nn\Phi(v)\>-\<\nn U(x),\nn^2\Phi(v)\nn u(x)\>\\
&\qquad\qquad-\int_{\R^d}\big(\<\nn\Phi(v),\nn^2u(x)\nn\Phi(v)\>-\<\nn U(x),\nn^2\Phi(v)\nn u(x)\>\big)\,\mu_2(\d v)\bigg)\\
&=-\frac{1}{c^\star}\big(\<\nn\Phi(v),\nn^2u(x)\nn\Phi(v)\>-c^\star\Delta u(x)\\
&\qquad\quad\,\,+c^\star\<\nn U(x),\nn u(x)\>-\<\nn U(x),\nn^2\Phi(v)\nn u(x)\>\big).
\end{split}
\end{equation*}
Subsequently,  in addition to \eqref{W18} and the basic inequality: $2ab\le a^2+b^2$ for all $a,b\ge0$,   we find  that for some constants $C_1,C_2>0,$
\begin{equation}\label{W20}
\begin{split}
 \|(B_{c^\star}\mathscr L_0(I-\pi))^* \pi  f\|_2^2&\le\frac{2}{(c^\star)^2}\mu_1(\varphi)  + 4 \Big(1+\frac{1}{(c^\star)^2}\mu_2\big(\|\nn^2\Phi\|^2\big)\Big)\mu_1\big(  |\nn U |^2|\nn u |^2\big )\\
 &\le C_1\mu_1(\varphi)+C_2\mu_1\big(  |\nn U |^2|\nn u |^2\big ),
 \end{split}
\end{equation}
where for all $x\in\R^d,$
\begin{equation*}
  \varphi(x):=\int_{\R^d}\big(\<\nn\Phi(v),\nn^2u(x)\nn\Phi(v)\>-c^\star\Delta u(x)\big)^2\,\mu_2 (\d v).
\end{equation*}
Recall that $\Phi(v)=\psi(|v|^2)$ and $\mu_2(\d v)=\e^{-\Phi(v)}\,\d v$. Thus,
\eqref{W18} and \eqref{W0}  yield   that
\begin{align*}
  \varphi(x)&=\int_{ \R^d}\<\nn\Phi(v),\nn^2u(x)\nn\Phi(v)\>^2\,\mu_2(\d v)-(c^\star\Delta u(x))^2\\
&=4\sum_{i,j,k,\ell=1}^d\int_{ \R^d}\psi'(|v|^2)^4v_iv_jv_kv_\ell\,\mu_2(\d v)\,\partial_{ij}u(x)\partial_{k \ell}u(x) -(c^\star\Delta u(x))^2\\
&=4\bigg(\sum_{ i, j=1}^d\int_{ \R^d}\psi'(|v|^2)^4v_i^2v_j^2\,\mu_2(\d v)\big(\partial_{ii}u(x)\partial_{jj}u(x)+2(\partial_{ij}u(x))^2\big)\\
&\quad\quad-2\sum_{i=1}^d\int_{ \R^d}\psi'(|v|^2)^4v_i^4 \,\mu_2(\d v)\,(\partial_{ii}u(x))^2\bigg)-(c^\star\Delta u(x))^2\\
&\le 2 \sum_{ i, j=1}^d\int_{ \R^d}\psi'(|v|^2)^4(v_i^4+v_j^4)\,\mu_2(\d v)\big(\partial_{ii}u(x)\partial_{jj}u(x)+2(\partial_{ij}u(x))^2\big),
\end{align*}
where $\partial_{ij} :=\frac{\partial^2}{\partial x_i\partial x_j}$. The quantitative estimate above, besides \eqref{W18},   implies that
\begin{align*}
  \varphi(x)
 \le  \frac{1}{4}  \mu_2(|\nn \Phi|^4)\big((\Delta u(x))^2+2\|\nn^2u(x)\|_{\rm HS}^2\big) \le C_3\|\nn^2u(x)\|^2
\end{align*}
for some constant $C_3>0,$ in which $\|\cdot\|_{\rm HS}$ represents the Hilbert-Schmidt norm. Accordingly,
by applying \cite[Proposition 5]{CHSG}, there exists a constant $C_4>0$ such that
\begin{equation}\label{W21}
\mu_1(\varphi)\le C_4\|\pi f\|_2^2.
\end{equation}

To handle the term $\mu_1\big(  |\nn U |^2|\nn u |^2\big )$, note from \cite[p.\ 1027; line -7]{CHSG} that there exist constants $C_5,C_6>0$ such that
\begin{equation*}
\mu_1\big(  |\nn U |^2|\nn u |^2\big )\le C_5\mu_1(\|\nn^2u\|^2)+C_6\mu_1(|\nn u|^2).
\end{equation*}
Thus, by invoking \cite[Proposition 5]{CHSG} and taking advantage of \cite[Corollary 30]{ADNR} (also see the arguments in \cite[p.\ 1027-1028]{CHSG} for \cite[Proposition 5]{CHSG}), there exists a constant $C_7>0$ such that
\begin{equation}\label{W22}
\mu_1\big(  |\nn U |^2|\nn u |^2\big )\le C_7\|\pi f\|_2^2.
\end{equation}

At length, the assertion \eqref{W23} is available by plugging \eqref{W21} and \eqref{W22} back into \eqref{W20}.
\end{proof}

In order to verify the Assumption {\bf(H4)}, we further need the following two lemmas.

\begin{lemma}\label{lem3.8}
Assume that $({\bf A}_U)$ and assumptions in Lemma $\ref{lem3.5}$ hold. Then, for any $f\in \mathfrak D$ and $d>2-\alpha,$
\begin{equation}\label{W27}
\begin{split}
\mathscr L_1^* B^*_{c^\star}\pi f (x,v)&=\frac{1}{c^\star} C_{d,2-\alpha}(d+\alpha-2)\e^{\Phi(v)}\,\mbox{\rm p.v.}\int_{\R^d}\frac{ \<v-\bar v,\nn^2\Phi(v)\nn u_{\pi f}(x)\>\e^{-\Phi(\bar v)}}{|v-\bar v|^{d+\alpha}}\,\d \bar v\\
&\quad -\frac{1}{c^\star}c_{d,\alpha}\e^{ \Phi(v)}\,\mbox{\rm p.v.}\int_{\R^d }\frac{ \<\nn \Phi(v)-\nn \Phi(\bar v),\nn u_{\pi f}(x)\>  \e^{-\Phi(\bar v)} }{|v-\bar v|^{d+\alpha}}\d \bar v,
\end{split}
\end{equation}
where
\begin{equation}\label{W25}
 C_{d,\alpha}:=\Gamma((d-\alpha)/2)/(2^\alpha\pi^{d/2}\Gamma(\alpha/2)),\quad
 c_{d,\alpha}:=2^\alpha\Gamma((d+\alpha)/2)/(\pi^{d/2}|\Gamma(-\alpha/2)|).
\end{equation}
\end{lemma}

\begin{proof}
In view of  $\pi^*=\pi$, $\mathscr L_0^*=-\mathscr L_0$ as well as   $\pi B_{c^\star} =B_{c^\star} $, we deduce from Lemma \ref{lem3.5} that
for any $f,g\in \mathfrak D,$
\begin{equation}\label{W24}
\begin{split}
\<B_{c^\star}f,g\>_2&= \<\pi  \mathscr L_0^* f,({c^\star}I-\pi  \mathscr L_0^2\pi)^{-1}\pi g\>_2\\
&= \frac{1}{{c^\star}}\<   \mathscr L_0^* f,(I-\mathscr L_{  OD})^{-1}\pi g\>_2,
\end{split}
\end{equation}
where in the second identity we also used the fact that $({c^\star}I-\pi  \mathscr L_0^2\pi)^{-1}\pi g$ is independent of the $v$-variable.
For $g\in \mathfrak D$, in terms of \cite[Proposition 4]{CHSG}, the Poisson equation
\begin{equation*}
(I-\mathscr L_{OD})u_{\pi g} =\pi g
\end{equation*}
has a unique classical solution $u_{\pi g}\in C_b^\8(\R^d)$.  Whereafter, we infer from \eqref{W24} that
\begin{equation*}
\<B_{c^\star}f,g\>_2
 = \frac{1}{{c^\star}}\<     f,\mathscr L_0u_{\pi g}\>_2.
\end{equation*}
 This, combining with the definition of $\mathscr L_0$, leads to
\begin{equation*}
B_{c^\star}^*f(x,v)=\frac{1}{c^\star}\<\nn \Phi(v),\nn u_{\pi f}(x)\>,\qquad f\in\mathfrak D.
\end{equation*}
Next, employing \eqref{W8} and \eqref{W10}, in addition to   the chain rule,   yields
\begin{equation}\label{W26}
\begin{split}
\mathscr L_1^*B_{c^\star}^*\pi f (x,v)=&- \frac{1}{c^\star}\e^{\Phi(v)}\mbox{div}_v\big(\<\nn\Phi(v),\nn u_{\pi f}(x)\>\nn \big((-\Delta)^{{\alpha}/{2}-1}\e^{-\Phi(v)}\big)\big) \\
&-\frac{1}{c^\star}\e^{\Phi(v)}(-\Delta_v  )^{{\alpha}/{2} }\big(\<\nn \Phi(v),\nn u_{\pi f}(x)\>\e^{-\Phi(v)}\big)\\
=&-  \frac{1}{c^\star}\e^{\Phi(v)} \big\<\nn^2\Phi(v)\nn u_{\pi f}(x), \nn  \big((-\Delta )^{{\alpha}/{2}-1}\e^{-\Phi(v)}\big)\big\> \\
&+ \frac{1}{c^\star}\e^{\Phi(v)}\big(\<\nn \Phi(v),\nn u_{\pi f}(x)\>(-\Delta)^{{\alpha}/{2}}\e^{-\Phi(v)}\\
&\qquad\quad\qquad- (-\Delta_v )^{{\alpha}/{2}}\big(\<\nn \Phi(v),\nn u_{\pi f}(x)\>\e^{-\Phi(v)}\big)\big)\\
=:&\phi_1(x,v)+\phi_2(x,v).
\end{split}
\end{equation}
Owing to  $d>2-\alpha$,
 it follows from \cite[Theorem 1.1]{K} that
\begin{equation*}
\begin{split}
\phi_1(x,v)
&= \frac{1}{c^\star} C_{d,2-\alpha}(d+\alpha-2)\e^{\Phi(v)}\,\mbox{p.v.}\int_{\R^d}\frac{ \<v-\bar v,\nn^2\Phi(v)\nn u_{\pi f}(x)\>\e^{-\Phi(\bar v)}}{|v-\bar v|^{d+\alpha}}\,\d \bar v,
\end{split}
\end{equation*}
and
\begin{equation*}
\begin{split}
\phi_2(x,v)= & \frac{1}{c^\star}c_{d,\alpha}\e^{\Phi(v)} \,\mbox{p.v.}\int_{\R^d }\frac{\<\nn \Phi(v),\nn u_{\pi f}(x)\>(\e^{-\Phi(v)} -\e^{-\Phi(\bar v)})}{|v-\bar v|^{d+\alpha}}\d \bar v\\
&-\frac{1}{c^\star}c_{d,\alpha}\e^{\Phi(v)} \,\mbox{p.v.}\int_{\R^d }\frac{\<\nn \Phi(v),\nn u_{\pi f}(x)\>\e^{-\Phi(v)}-\<\nn \Phi(\bar v),\nn u_{\pi f}(x)\>\e^{-\Phi(\bar v)}}{|v-\bar v|^{d+\alpha}}\d \bar v\\
=&-\frac{1}{c^\star}c_{d,\alpha}\e^{ \Phi(v)}\,\mbox{p.v.}\int_{\R^d }\frac{ \<\nn \Phi(v)-\nn \Phi(\bar v),\nn u_{\pi f}(x)\>  \e^{-\Phi(\bar v)} }{|v-\bar v|^{d+\alpha}}\d \bar v,
\end{split}
\end{equation*}
where $C_{d,2-\alpha}$ and $c_{d,\alpha}$ were defined in \eqref{W25}. Thus, substituting the explicit expressions above on $\phi_1$ and $\phi_2$
into \eqref{W26} yields the desired assertion \eqref{W27}.
\end{proof}

\begin{lemma}\label{lem3.6}
Assume that $\Psi\in C^3(\R^d;\R_+)$ such that $C_\Psi:=\int_{\R^d}  \e^{-\Psi(u)}\,\d u<\8$ and $\|\nn \Psi\|_\8<\8$. For $\beta,\gamma>0$ and $v,y\in \R^d$, set
\begin{equation}\label{W19}
\begin{split}
\Psi_{\beta,\gamma}(v, y):& =\beta\e^{\Psi(v)}\,{\rm p.v. }\int_{\R^d}\frac{ \<u,\nn^2\Psi(v)  y\>\e^{-\Psi(  v-u)}}{|u|^{d+\alpha}}\,\d u \\
&\quad  -\gamma\e^{ \Psi(v)}\,{\rm p.v. }\int_{\R^d }\frac{ \<\nn \Psi(v)-\nn \Psi( v-u), y\>  \e^{-\Psi(  v-u)} }{|u|^{d+\alpha}}\d u.
\end{split}
\end{equation}
Then,
$\Psi_{\beta,\gamma}(v, y)$ is well defined so that for any $v,y\in \R^d$,
$|\Psi_{\beta,\gamma}(v, y)|\le \Theta(v)|y|$, where $v\mapsto \Theta(v)$ is positive and locally bounded on $\R^d$.
Assume further  that    $\Psi\in C^3(\R^d;\R_+)$ is a  radial function   so that
$|v|\mapsto \Psi(v)=\Psi(|v|)$ is non-decreasing
and  there exist constants $c^*,v^*>0$ such that for all $v\in\R^d$ with $|v|\ge v^*,$
\begin{equation}\label{W13}
 \sup_{u\in B_1(v)}\|\nn^i\Psi (u)\|\le c^*\|\nn^i\Psi (v)\|,\quad i=1,2,3.
\end{equation}
Then,
there exists a constant $c_0>0$ such that for all $v\in\R^d$ with $|v|\ge v^*$ and $y\in\R^d$,
\begin{equation}\label{W14}
\begin{split}
|\Psi_{\beta,\gamma}(v, y)|\le c_0 \tilde\Psi_{\beta,\gamma}(v)|y|
\end{split}
\end{equation}
where
\begin{align*}\tilde\Psi_{\beta,\gamma}(v):= &\big\|\nn^2\Psi(v)\big\|\e^{\Psi(v)}\Big[ \e^{-\Psi(v)}|\nn\Psi(v)|+  |v|^{-(d+\alpha-1)}  \\
&\qquad\qquad\qquad\quad+ (\I_{\{\alpha\in(1,2)\}}+\I_{\{\alpha=1\}}\log|v|+|v|^{1-\alpha}\I_{\{\alpha\in (0,1)\}})\,\e^{-\Psi(v/2)}\Big]\\
&+  |\nn\Psi(v)|\,\|\nn^2\Psi(v) \|+  \|\nn^3\Psi(v )  \| +\e^{ \Psi(v)}\big( |v|^{-(d+\alpha)}   +\e^{-\Psi(v/2)} \big).\end{align*}

In particular, if  \begin{equation}\label{D7-}\sup_{v\in \R^d} \left(\|\nabla^2\Psi(v)\||v|\right)<\infty,\quad  \|\nn^3\Psi\|_\8<\8, \quad \sup_{v\in \R^d} |\Psi(v)-\Psi(v/2)|<\infty\end{equation} and the integrability
\begin{equation}\label{D7--}\int_{\R^d}
\frac{\e^{\Psi(v)}}{(1+|v|)^{2(d+\alpha)}}\,\d v<\8\end{equation} hold respectively,  then $\mu_2(\Theta^2)<\8$.
\end{lemma}
\begin{proof} The proof is split into three parts.

(i) For $\Psi_{\beta,\gamma}$ introduced  in \eqref{W19},
via   change of variables, it holds that
\begin{equation*}
\begin{split}
\Psi_{\beta,\gamma}(v,y)&=-\beta\e^{\Psi(v)}\,{\rm p.v. }\int_{\R^d}\frac{ \<u,\nn^2\Psi(v)y\>\e^{-\Psi(  v+u)}}{|u|^{d+\alpha}}\,\d u\\
&\quad-\gamma\e^{ \Psi(v)}\,{\rm p.v. }\int_{\R^d }\frac{ \<\nn \Psi(v)-\nn \Psi( v+u),y\>  \e^{-\Psi(  v+u)} }{|u|^{d+\alpha}}\d u.
\end{split}
\end{equation*}
Thus, we have
\begin{equation}\label{e:eq-01}\begin{split}
\Psi_{\beta,\gamma}(v,y)&=\frac{1}{2}\beta\e^{\Psi(v)}\int_{\R^d}\frac{ \<u,\nn^2\Psi(v)y\>}{|u|^{d+\alpha}}\big(\e^{-\Psi(v-u)}-\e^{-\Psi(v+u)}\big)\,\d u\\
&\quad-\frac{1}{2}\gamma\e^{ \Psi(v)}\int_{\R^d}\frac{1}{|u|^{d+\alpha}}\big(\<\nn \Psi(v)-\nn \Psi( v-u), y\>\e^{-\Psi(v-u)}\\
&\qquad\qquad\qquad\qquad\qquad\quad+ \<\nn \Psi(v)-\nn \Psi( v+u),y\>\e^{-\Psi(v+u)}\big) \,\d u\\
&=:\frac{1}{2}\beta\e^{\Psi(v)}I_1(v,y)-\frac{1}{2}\gamma\e^{ \Psi(v)}I_2(v,y).
\end{split}\end{equation}

Notice  that
\begin{equation}\label{e:eq-02}
\begin{split}
|I_1(v,y)|&\le \|\nn^2\Psi(v)\| |y|\int_{\{|u|\ge 1\}}\frac{ 1}{|u|^{d+\alpha-1}}\big(\e^{-\Psi(v-u)}+\e^{-\Psi(v+u)}\big)\,\d u\\
&\quad +\|\nn^2\Psi(v)\| |y|\int_{\{|u|\le 1\}}\frac{1}{|u|^{d+\alpha-1}}\big|\e^{-\Psi(v-u)}-\e^{-\Psi(v+u)}\big|\,\d u \\
&=:I_{11}(v,y)+I_{12}(v,y).
\end{split}
\end{equation}
It is easy to see that  $$I_{11}(v,y)\le 2\, C_\Psi\, \|\nn^2\Psi(v)\|\, |y|$$
via change of variables, and that    $$I_{12}(v,y)\le 2\bigg(\int_{\{|u|\le 1\}}\frac{1}{|u|^{d+\alpha-2}}\,\d u\bigg) \|\nn^2\Psi(v)\| \| \nabla  \Psi   \|_\8 |y| $$
by the mean value theorem and $\Psi\ge0$.

On the other hand, it is obvious that
\begin{equation}\label{e:eq-03}
\begin{split}
I_2(v,y)&=\int_{\{|u|\le 1\}}\frac{1}{|u|^{d+\alpha}} \Psi_{1}(v,u,y)  \,\d u+\int_{\{|u|\le 1\}}\frac{1}{|u|^{d+\alpha}} \Psi_{2}(v,u,y)  \,\d u\\
&\quad+\int_{\{|u|\ge 1\}}\frac{1}{|u|^{d+\alpha}} \big(\Psi_{1}(v,u,y) + \Psi_{2}(v,u,y)\big)\,\d u,
\end{split}
\end{equation}
where
\begin{equation*}
\begin{split}
\Psi_{1}(v,u,y):&=\<\nn \Psi(v)-\nn \Psi( v-u), y\>(\e^{-\Psi(v-u)}-\e^{-\Psi(v+u)} ),\\
\Psi_{2}(v,u,y):&=\< 2\nn\Psi(v)-\nn\Psi(v-u)-\nn\Psi(v+u),y\>\e^{-\Psi(v+u)}.
\end{split}
\end{equation*}

Applying  the mean value theorem, besides $\Psi\ge0,$ yields
$$\int_{\{|u|\le 1\}}\frac{1}{|u|^{d+\alpha}} \Psi_{1}(v,u,y)  \,\d u\le 2\|\nabla \Psi\|^2_\8\bigg(\int_{\{|u|\le 1\}}\frac{1}{|u|^{d+\alpha-2}}\,\d u\bigg)|y|.$$
Furthermore, via change of variables again, it is ready to see that
$$\int_{\{|u|> 1\}}\frac{1}{|u|^{d+\alpha}} \Psi_{1}(v,u,y)  \,\d u\le 2C_\Psi\|\nn\Psi\|_\8 |y|.$$

With the aid of the facts that
\begin{equation*}
\nn\Psi(v+u)=\nn\Psi(v )+\nn^2\Psi(v ) u+\int_0^1\int_0^s  \nn \big(\nn^2\Psi(v+\theta u )u\big)u\,\d \theta\,\d s
\end{equation*}
and
\begin{equation*}
\nn\Psi(v-u)=\nn\psi(v )-\nn^2\Psi(v ) u+\int_0^1\int_0^s \nn\big(\nn^2\Psi(v-\theta u )u\big)u\,\d \theta\,\d s,
\end{equation*}
we find that
\begin{equation}\label{e:eq-04}
\begin{split}
 &2\nn\Psi(v )-\nn\Psi(v+u)-\nn\Psi(v-u)\\
 &=- \int_0^1\int_0^s\big( \nn \big(\nn^2\Psi(v+\theta u )u\big)u+ \nn\big(\nn^2\Psi(v-\theta u )u\big)u\big)\,\d \theta\,\d s.
\end{split}
\end{equation}
This obviously implies that
$$\int_{\{|u|\le 1\}}\frac{1}{|u|^{d+\alpha}} \Psi_{2}(v,u,y)  \,\d u\le 2\Big(\sup_{z\in B_1(v)}\|\nabla^3 \Psi(z)\|\Big)\left(\int_{\{|u|\le 1\}}\frac{1}{|u|^{d+\alpha-2}}\,\d u\right)|y|.$$ Furthermore, it is obvious that
$$\int_{\{|u|> 1\}}\frac{1}{|u|^{d+\alpha}} \Psi_{2}(v,u,y)  \,\d u\le 4C_\Psi\|\nn\Psi\|_\8  |y|.$$

Putting all the estimates above into \eqref{e:eq-01} and taking $\|\nn\Psi\|_\8<\8$, we obtain that $\Psi_{\beta,\gamma}(v, y)$ is well defined so that for any $v,y\in \R^d$,
$|\Psi_{\beta,\gamma}(v, y)|\le \Theta(v)|y|$, where $v\mapsto \Theta(v)$ is positive and locally bounded on $\R^d$.

(ii) In this part, we shall fix $v\in\R^d$ with $|v|\ge v^* $ and $y\in\R^d$. Let $I_{11}(v,y)$ and $I_{12}(v,y)$ be those defined in \eqref{e:eq-02}.
Taking  the  non-decreasing property of $\Psi$ into consideration,
  we derive   that
\begin{align*}
I_{11}(v,y)
&\le2\big\|\nn^2\Psi(v)\big\||y|\bigg(\int_{\{|u|\ge1\}\cap\{\<u,v\>\ge0\}\cap\{|v-u|\le\frac{1}{2}|v|\}}\frac{1}{|u|^{d+ \alpha-1}}  \e^{-\Psi(v-u)} \,\d u\\
&\qquad\qquad\qquad\qquad+\int_{\{|u|\ge1\}\cap\{\<u,v\>\le0\}\cap\{|v+u|\le\frac{1}{2}|v|\}}\frac{1}{|u|^{d+ \alpha-1}}  \e^{-\Psi(v+u)} \,\d u\\
&\qquad\qquad\qquad\qquad+\int_{\{|u|\ge1\}\cap\{\<u,v\>\ge0\}\cap\{|v-u|\ge\frac{1}{2}|v|\}\cap\{|u|\ge |v|\}}\frac{1}{|u|^{d+ \alpha-1}}  \e^{-\Psi(v-u)} \,\d u\\
&\qquad\qquad\qquad\qquad+\int_{\{|u|\ge1\}\cap\{\<u,v\>\le0\}\cap\{|v+u|\ge\frac{1}{2}|v|\}\cap\{|u|\ge |v|\}}\frac{1}{|u|^{d+ \alpha-1}}  \e^{-\Psi(v+u)} \,\d u\\
&\qquad\qquad\qquad\qquad+\int_{\{|u|\ge1\}\cap\{\<u,v\>\ge0\}\cap\{|v-u|\ge\frac{1}{2}|v|\}\cap\{|u|\le |v|\}}\frac{1}{|u|^{d+ \alpha-1}}  \e^{-\Psi(v-u)} \,\d u\\
&\qquad\qquad\qquad\qquad+\int_{\{|u|\ge1\}\cap\{\<u,v\>\le0\}\cap\{|v+u|\ge\frac{1}{2}|v|\}\cap\{|u|\le |v|\}}\frac{1}{|u|^{d+ \alpha-1}}  \e^{-\Psi(v+u)} \,\d u\bigg)\\
&\le 8\big\|\nn^2\Psi(v)\big\||y|\bigg[\frac{ C_\Psi}{(|v|/2)^{d+\alpha-1}}+ \bigg(\I_{\{1<\alpha<2\} } \int_{\{|u|\ge1\}}\frac{1}{|u|^{d+ \alpha-1}}\,\d u  \\
&\qquad\qquad\qquad\qquad\qquad\qquad\qquad\qquad+\I_{\{0<\alpha\le1\}} \int_{\{1\le |u|\le|v|\}}\frac{1}{|u|^{d+\alpha-1}}\,\d u\bigg) \, \e^{-\Psi(v/2)}     \bigg],
\end{align*}
 where  the last  display is valid
due to   $|u|\ge |v|-|v\pm u|\ge \frac{1}{2}|v|$ in case of $|v\pm u|\le \frac{1}{2}|v|$.

In view of $\|\nn\Psi\|_\8<\8$ and \eqref{W13}, we obviously have for some constant $C_0>0,$
 \begin{equation*}
 \sup_{|v|\ge v^*}\sup_{u\in B_1(v)}\bigg(\e^{\Psi(v)-\Psi(u)}\frac{|\nn\Psi(u)|}{|\nn \Psi(v)|}\bigg)\le C_0.
 \end{equation*}
With the help of this estimate,  we arrive at
 \begin{equation*}
 I_{12}(v,y) \le 2C_0\bigg(\int_{\{|u|\le 1\}}\frac{1}{|u|^{d+\alpha-2}}\,\d u\bigg)\big\|\nn^2\Psi(v)\big\||\nn \Psi(v)| \e^{-\Psi(v)}|y|.
 \end{equation*}

Subsequently, according to the estimates for $I_{11}(v,y)$ and $I_{12}(v,y)$,
 we conclude  that there exists a constant $C_1>0$ so that
  for all $v,y\in \R^d$ with $|v|\ge v^*$,
 \begin{equation}\label{W17}
 \begin{split}
 | I_1(v,y)|\le  C_1 \big\|\nn^2\Psi(v)\big\| \Big[ & |\nn \Psi(v)| \e^{-\Psi(v)}  + |v|^{-(d+\alpha-1)}  \\
 &+  \big(\I_{\{\alpha\in(1,2)\}}+\I_{\{\alpha=1\}}\log|v|+|v|^{1-\alpha}\I_{\{\alpha\in (0,1)\}}\big)\, \e^{-\Psi(v/2)} \Big]|y|.
 \end{split}
 \end{equation}

In the sequel, let $I_{21}(v,y)$, $I_{22}(v,y)$ and $I_{23}(v,y)$ be the those terms on the right hand side of \eqref{e:eq-03}. Below,    we  aim to treat them separately.
By invoking  $\|\nn\Psi\|_\8<\8$ and \eqref{W13}, we obtain that for some constant $C_2>0,$
 \begin{equation*}
 \begin{split}
  \big|I_{21}(v,y)\big|
 &=\bigg| \int_{\{|u|\le 1\}}\int_0^1\int_0^1\frac{1}{|u|^{d+\alpha}} \<\nn^2\Psi(v-su)u,y\>\\
 &\qquad\qquad\times\big(\e^{-\Psi(v-\theta u)}\<\nn \Psi(v-\theta u),u\>
 -\e^{-\Psi(v+\theta u)}\<\nn \Psi(v+\theta u),u\> \big )\,\d s\,\d\theta\,\d u\bigg|\\
 &\le C_2\bigg(\int_{\{|u|\le 1\}}\frac{1}{|u|^{d+\alpha-2}}\,\d u\bigg)\,\e^{-\Psi(v)}|\nn\Psi(v)|\,\|\nn^2\Psi(v) \|\,|y|.
 \end{split}
\end{equation*}
Next, with the aid of \eqref{e:eq-04}, we can get from  $\|\nn\Psi\|_\8<\8$ and \eqref{W13}  that for some constant $C_3>0,$
\begin{equation*}
\begin{split}
\big|I_{22}(v,y)\big|& =\bigg|\int_{\{|u|\le 1\}}\int_0^1\int_0^s\frac{1}{|u|^{d+\alpha}}\big( \big\<\nn\big(\nn^2\Psi(v+\theta u )u+\nn^2\Psi(v-\theta u )u\big)u,y\big\>\\
&\qquad\qquad\qquad\qquad\qquad\quad\,\,\times  \e^{-\Psi(v+u)}\,\d \theta\,\d s\,\d u\bigg|\\
&\le C_3\bigg(\int_{\{|u|\le 1\}}\frac{1}{|u|^{d+\alpha-2}}\,\d u\bigg)\,\e^{-\Psi(v)}\|\nn^3\Psi(v )  \| |y|.
\end{split}
\end{equation*}
  According to the definitions of $\Psi_1$ and $\Psi_2$, along with the non-decreasing property of the mapping $|v|\mapsto \Psi(v)=\Psi(|v|)$, we readily  deduce  that
\begin{align*}
 |I_{23}(v,y)|
 &\le6\|\nn\Psi\|_\8|y|\int_{\{|u|\ge1\}\cap\{\<u,v\>\ge0\}}\frac{1}{|u|^{d+\alpha}}    \e^{-\Psi(v-u)} \,\d u\\
 &\quad+6\|\nn\Psi\|_\8|y|\int_{\{|u|\ge1\}\cap\{\<u,v\>\le0\} }\frac{1}{|u|^{d+\alpha}}\e^{-\Psi(v+u)}\,\d u\\
&\le     6\|\nn\Psi\|_\8|y|\bigg(\int_{\{|u|\ge1\}\cap\{|v-u|\le\frac{1}{2}|v|\}}\frac{1}{|u|^{d+\alpha}}  \e^{-\Psi(v-u)} \,\d u\\
 &\qquad\qquad\qquad\quad+\int_{\{|u|\ge1\}\cap\{|v+u|\le\frac{1}{2}|v|\}}\frac{1}{|u|^{d+\alpha}}     \e^{-\Psi(v+u)} \,\d u\bigg)\\
&\quad+ 6\|\nn\Psi\|_\8|y|\bigg(\int_{\{|u|\ge1\}\cap\{|v-u|\ge\frac{1}{2}|v|\}}\frac{1}{|u|^{d+\alpha}}     \e^{-\Psi(v-u)} \,\d u\\
 &\qquad\qquad\qquad\quad+\int_{\{|u|\ge1\}\cap\{|v+u|\ge\frac{1}{2}|v|\}}\frac{1}{|u|^{d+\alpha}}   \e^{-\Psi(v+u)} \,\d u\bigg).
\end{align*}
Consequently, by utilizing  the fact that $|u|\ge |v|-|v\pm u|\ge \frac{1}{2}|v|$ as long as  $|v\pm u|\le \frac{1}{2}|v|$, we derive from $C_\Psi<\8$ and $\|\nn\Psi\|_\8$ that there exists a  constant   $C_4 >0 $ such that
\begin{equation*}
\begin{split}
 |I_{23}(v,y)|\le C_4\big(   |v|^{-(d+\alpha)}        +   \e^{-\Psi(v/2)}\big)|y|.
 \end{split}
\end{equation*}
This, in addition to the estimates concerned with $I_{21}(v,y)$ and $I_{22}(v,y)$, leads to
\begin{equation}\label{W16}
\begin{split}
|I_2(v,y)|&\le C_5\e^{-\Psi(v)}\big(|\nn\Psi(v)|\,\|\nn^2\Psi(v) \|  + \|\nn^3\Psi(v )\| \big)|y|  \\
&\quad+ C_6   \big(   |v|^{-(d+\alpha)}        +   \e^{-\Psi(v/2)}\big)|y|
\end{split}
\end{equation}
for some constants $C_5,C_6>0$.

Therefore, the desired assertion \eqref{W14} follows by combining \eqref{W17} with \eqref{W16}.

(iii) Finally, since $\Theta(v)$ is locally bounded, the last assertion follows from \eqref{W14} (in particular, the estimate for $\tilde\Psi_{\beta,\gamma}(v)$),   \eqref{D7-} and \eqref{D7--}, as well as some calculations.
\end{proof}

\begin{proposition}\label{P4}
Assume that $({\bf A}_U)$ and $({\bf A}_\Phi)$ are satisfied. Then, the Assumption $({\bf H}_4)$ holds true.
\end{proposition}

\begin{proof}
To show the Assumption $({\bf H}_4)$, it is sufficient to prove respectively that there exist constants $c_1,c_2>0$ such that
for all $f\in \mathfrak D$,
\begin{equation}\label{W3}
|\<B_{c^\star}\mathscr L_0(I-\pi)f,f\>_2|\le c_1\|\pi f\|_2\|(I-\pi)f\|_2
\end{equation}
and
\begin{equation}\label{W4}
|\<B_{c^\star}\mathscr L_1(I-\pi)f,f\>_2|\le c_2\|\pi f\|_2\|(I-\pi)f\|_2.
\end{equation}

Due to $\pi B_{c^\star}=B_{c^\star}$ and $(I-\pi)^2=I-\pi$, it is easy to see that
\begin{equation*}
\begin{split}
\big|\<B_{c^\star}\mathscr L_0(I-\pi)f,f\>_2\big|&=\big|\<(I-\pi)f,(B_{c^\star}\mathscr L_0(I-\pi))^*\pi f\>_2\big|\\
&\le \|(I-\pi)f\|_2\|(B_{c^\star}\mathscr L_0(I-\pi))^*\pi f\|_2.
\end{split}
\end{equation*}
Whence,  \eqref{W3} follows from Lemma \ref{lem3.7}.

Again, by means of $\pi B_{c^\star}=B_{c^\star}$, we infer that
 for all $f\in\mathfrak D$,
\begin{equation*}
\<B_{c^\star}\mathscr L_1(I-\pi)f,f\>_2=\< (I-\pi)f, \mathscr L_1^* B^*_{c^\star}\pi f\>_2
\le  \|(I-\pi)f\|_2\| \mathscr L_1^* B^*_{c^\star}\pi f\|_2.
\end{equation*}
Next, taking  Lemma \ref{lem3.8} into consideration gives that
\begin{equation*}
\begin{split}
\mathscr L_1^* B^*_{c^\star}\pi f (x,v)&=\frac{1}{c^\star} C_{d,2-\alpha}(d+\alpha-2)\e^{\Phi(v)}\,\mbox{\rm p.v.}\int_{\R^d}\frac{ \<u,\nn^2\Phi(v)\nn u_{\pi f}(x)\>\e^{-\Phi(v-u)}}{|u|^{d+\alpha}}\,\d u\\
&\quad -\frac{1}{c^\star}c_{d,\alpha}\e^{ \Phi(v)}\,\mbox{\rm p.v.}\int_{\R^d }\frac{ \<\nn \Phi(v)-\nn \Phi( v-u),\nn u_{\pi f}(x)\>  \e^{-\Phi(v-u)} }{|u|^{d+\alpha}}\d u.
\end{split}
\end{equation*}
Subsequently, applying Lemma \ref{lem3.6} with $\beta=\frac{1}{c^\star} C_{d,2-\alpha}(d+\alpha-2)$, $\gamma=\frac{1}{c^\star}c_{d,\alpha}$ and $y=\nn u_{\pi f}(x)$ yields that for some constant $c_0>0$,
\begin{equation*}
|\mathscr L_1^* B^*_{c^\star}\pi f (x,v)|\le c_0\Theta(v) |\nn u_{\pi f}(x)|,
\end{equation*}where $\Theta(v)$ is given in Lemma \ref{lem3.6}.
Consequently, \eqref{W4} is valid  thanks to $\mu( \Theta^2)<\8$ and \cite[Corollary 30]{ADNR}.
\end{proof}

\begin{remark}\label{Rem:add1}To verify Assumption $({\bf H}_4)$,  we turn to the inequalities \eqref{W3} and \eqref{W4}. As for \eqref{W3}, we make use of the regularity properties of the Poisson equation \eqref{D12} associated with the Hamiltonian operator (i.e., the anti-symmetric part) $\mathscr L_0$ given in \eqref{W2}. The approach is inspired by the previous work in the Brownian motion setting; see \cite{CHSG}. However, to obtain \eqref{W4} it is extremely non-trivial. Note that, since the operator $\mathscr L_1$ is not only non-local but also  non-symmetric, the expression  \eqref{W27} for the dual operator $\mathscr L_1^* B^*_{c^\star}$ is a little bit complex, and, in particular, $\mathscr L_1^* B^*_{c^\star}$ does not enjoy the chain rule property. On the other hand, in order to establish the bound for $\| \mathscr L_1^* B^*_{c^\star}\|_{2\to 2}$, we need some explicit estimates as stated in Lemma \ref{lem3.6}, which in turn require the boundedness condition \eqref{D7-} and
the integrability condition \eqref{D7--}.
 This partly points out the reason that why we impose  $\beta<2\alpha$ in Corollary \ref{cor}.
\end{remark}

With the previous preparations at hand, we are in position to
complete the
\begin{proof}[Proof of Theorem $\ref{thm1}$]
Since $x\mapsto \e^{-U(x)}$ and $v\mapsto \e^{-\Phi(v)}$
are integrable, $C_U,C_\Phi\in(0,\8)$. Due to the uniform boundedness of $\nn\Phi$ and $\nn^2\Phi$ (see $({\bf A}_{\Phi,2})$), we have $\mu_2( |\nn \Phi|^4+\|\nn^2\Phi\|^2)<\8$ and $0<\mu_2(|\nn \Phi|^2)<\8$ right now.
On the other hand, according to the function that $\Phi$ is radial and the boundedness of $|\nn \Phi|$ as well as the fact that $v\mapsto \e^{-\Phi(v)}$
is integrable, $\lim_{|v|\to\8}|\nn\e^{-\Phi(v)}|=0
$. In particular, \eqref{D3} holds.
Under the Assumption $({\bf A}_{U,2})$, $\mu_1$ satisfies the Poincar\'{e} inequality \eqref{E4} (see e.g. \cite[Corollary 1.6]{BBCG}); under the Assumption $({\bf A}_{\Phi,4})$, $\mu_2$ satisfies the Poincar\'{e} inequality \eqref{E14} (see e.g. \cite[Theorem 1.1]{WW}). Therefore, all the assumptions imposed on Propositions \ref{P1} and \ref{P3} are fulfilled. Furthermore, under the Assumptions $({\bf A}_{U})$-$({\bf A}_{\Phi})$, all preconditions in Lemmas \ref{lem3.5}-\ref{lem3.8} are satisfied so that  Proposition \ref{P4} is available. Thus,
the proof of Theorem \ref{thm1} is finished by applying Theorem \ref{thm} and
taking Propositions \ref{P1}, \ref{P3} and \ref{P4} into account.
\end{proof}

In the end, we finish the
\begin{proof}[Proof of Corollary $\ref{cor}$] According to the expression of $\Phi$, we have $\Phi\in C^3(\R^d;\R_+)$,
  $\psi(r)=\frac{1}{2}(d+\beta)\log(1+r)$, $r\ge0,$ which is non-decreasing, and $v\mapsto\e^{-\Phi(v)}$ is integrable. Hence, the Assumption $({\bf A}_{\Phi,1})$ is verified. Again, in terms of the form of $\Phi$, we find that
 \begin{equation*}
  \Phi(v)-\Phi(v/2) =\frac{1}{2}(d+\beta)\big(\log(1+|v|^2 )-\log(1+|v|^2/4 )\big)\le (d+\beta)\log 2.
 \end{equation*}
 Next, note that
  \begin{equation}\label{EEE}
\nn\Phi(v)=\frac{(d+\beta)v}{1+|v|^2},\qquad \nn^2\Phi(v)=(d+\beta)\bigg(\frac{1}{1+|v|^2}I_{d\times d}-\frac{2(v\otimes v)}{(1+|v|^2)^2}\bigg).
\end{equation}
Thus, all   assumptions in $({\bf A}_{\Phi,2})$ are satisfied.  Due to $\beta<2\alpha$, we deduce that
 $$\int_{\R^d} \frac{\e^{\Phi(v)}}{(1+|v|)^{2(d+\alpha)}}\,\d v<\infty.$$
This, together with \eqref{EEE}, leads to the fulfillment of the Assumption $({\bf A}_{\Phi,3})$. At last, thanks to $\beta\ge\alpha$,
we conclude that the Assumption $({\bf A}_{\Phi,4})$ is available. Therefore, the proof is complete.
\end{proof}

\ \

\noindent \textbf{Acknowledgements.}
The research is supported by the National Key R\&D Program of China (2022YFA1006000) and  NSF of China (Nos.\ 11831014, 12071340, 12071076 and 12225104).

\end{document}